\documentclass[11pt,a4paper,reqno]{amsart}
\usepackage{amsaddr}
\usepackage{amsmath,amsfonts,amssymb,amsthm,mathrsfs}
\usepackage[utf8]{inputenc}
\usepackage[T1]{fontenc}
\usepackage{lmodern}
\usepackage{latexsym}
\usepackage{cancel}
\usepackage{microtype}
\usepackage{caption}
\usepackage{MnSymbol}
\usepackage{xcolor}
\usepackage{enumerate}
\usepackage[normalem]{ulem}
\usepackage{bbm}
\usepackage{geometry}
\usepackage{marginnote}

\newcommand{\BB}{{\cal B}}

\newcommand{\EE}{{\cal E}}
\newcommand{\FF}{{\cal F}}

\newcommand{\MM}{{\cal M}}

\newcommand{\WW}{{\cal W}}

\newcommand{\BR}{{\mathbb R}}
\newcommand{\BX}{{\mathbb X}}

\newcommand{\cal}{\mathcal}

\usepackage[colorlinks=true,linkcolor=blue,citecolor=magenta]{hyperref}

\usepackage{graphics,tikz}

\newtheorem{theorem}{\bf Theorem}[section]
\newtheorem{proposition}[theorem]{\bf Proposition}
\newtheorem{lemma}[theorem]{\bf Lemma}
\newtheorem{corollary}[theorem]{\bf Corollary}

\theoremstyle{definition}
\newtheorem{definition}[theorem]{Definition}

\newtheorem{remark}[theorem]{Remark}

\numberwithin{equation}{section}

\usepackage{setspace}

\begin{document}

\title[Asymptotics for logistic-type equations]{Asymptotics for logistic-type equations with Dirichlet
fractional Laplace operator}

\author{Tomasz Klimsiak}

\address{Institute of Mathematics, Polish Academy of Sciences,
ul. \'{S}niadeckich 8,   00-656 Warsaw, Poland, \and Faculty of
Mathematics and Computer Science, Nicolaus Copernicus University,
Chopina 12/18, 87-100 Toruń, Poland, \\e-mail: {\tt tomas@mat.umk.pl}}

\begin{abstract}
We study the asymptotics of solutions of logistic type equations with fractional Laplacian
as  time goes to infinity  and as the exponent in nonlinear part goes to infinity.
We prove strong convergence of solutions in the energy space and uniform convergence
to the solution of an obstacle problem. As a by-product, we also prove the cut-off property for eigenvalues of  the Dirichlet fractional Laplace operator  perturbed by exploding potentials.
\end{abstract}

\maketitle

\section{Introduction}

Let $D\subset \BR^d$ ($d\ge 2$) be a bounded Lipschitz domain, $\varphi, b$ be  bounded positive Borel measurable functions on $D$ and $a>0$. In the present paper, we investigate asymptotics, as $p\rightarrow \infty$ and $t\rightarrow \infty$, of solutions to the following Cauchy-Dirichlet   problem
\begin{equation}
\label{eq1.3} \left\{
\begin{array}{l}\frac{d v_p}{d t}-\Delta^{\alpha} v_p= av_p-bv_p^p,\quad\mbox{in}\,\,\, D\times (0,\infty),\medskip\\
\,v_p=0,\quad\mbox{in}\,\,\, (\BR^d\setminus D)\times (0,\infty),
\medskip \\
\,v_p(0,\cdot)=\varphi,\quad\mbox{in}\,\,\, D.
\end{array}
\right.
\end{equation}
where $\alpha\in (0,1)$, and $\Delta^{\alpha}$ is the  fractional Laplacian (see Section \ref{sec2} for details).
Equations and systems of type (\ref{eq1.3}) serve as basic models in population biology. In classical models, $\alpha=1$,  the operator involved in (\ref{eq1.3}) is the usual Laplace operator. In the present paper, we concentrate on the study of (\ref{eq1.3}) with nonlocal operators, $\alpha\in (0,1)$.
In recent years, nonlocal population models attracted quite a lot  interest  (see \cite{BZ,CCR,MV,MPV,SLW} and the references therein).
They are designed to describe the  nonlocal dispersal strategy of animals.
This type of  dispersal strategy based on L\'evy flights  has been observed in nature (see, e.g., \cite{HQD,MV} for a discussion of this problem). Very recently, Caffarelli, Dipierro and Valdinoci \cite{CDV} investigated the existence problem for steady-state population model of type (\ref{eq1.3}) with additional nonlocal term on the right-hand side describing the nonlocal character of the species rate.

In the case of the classical Laplace operator, Dancer and Du \cite{DD1,DD2} proved a very interesting result stating   that
for large $p\ge 1$ the solutions of a stationary  counterpart to  (\ref{eq1.3}) behave like solutions of certain  steady-state predator-pray models. This common behaviour was described by  certain free boundary problem.

In the present paper, motivated by the results of Dancer and Du, we study the  asymptotic behaviour of solutions to (\ref{eq1.3}). We consider the following two cases:
\begin{enumerate}
\item[(i)] we pass to the limit in (\ref{eq1.3}) with $p\rightarrow +\infty$ and then with $t\rightarrow +\infty$;
\item[(ii)] we pass to the limit in (\ref{eq1.3}) with $t\rightarrow+ \infty$ and then with $p\rightarrow +\infty$.
\end{enumerate}
The most interesting part is the convergence as $p\rightarrow +\infty$ because by the known results  for the usual Laplace operator (see \cite{BM1,BM2,DO,DDM1,RT}), it is reasonable to expect that  the limit function is a solution of some free boundary problem (or, equivalently, the obstacle problem). 
This phenomenon was studied for the first time by Boccardo and  Murat  \cite{BM1} in the case of  equations with Leray-Lions type operator and with $a=0$, $b=1$. An interesting part is also the convergence as $t\rightarrow \infty$ in (i). It implies the  large-time asymptotics for an evolution obstacle problem and,  at the same time, an existence result for a stationary obstacle problem.
Asymptotics of solutions to equations  of type (\ref{eq1.3}) with classical Laplacian and general $a,b$ was investigated in \cite{DDM1,RT}.
To our knowledge,
there are no asymptotics results for (\ref{eq1.3}) with $\alpha\in(0,1)$ when $p\rightarrow \infty$ in (i) and (ii), and in (i) when $t\rightarrow \infty$. 

In the whole paper, we assume that the following hypotheses are satisfied.
\begin{enumerate}
\item[(H1)] $b\in \BB_b^+(D)$ is  nontrivial (i.e. $\int_D b\, dm>0$, where $m$ is the Lebesgue measure), there exists a  Lipschitz domain $D_0\subset D$ such that $\{b=0\}=\overline D_0$, and for every compact $K\subset D\setminus \overline D_0$,
\[
\inf_{x\in K} b(x)>0.
\]
\item[(H2)] $\varphi\in\BB^+_b(D)$ is nontrivial and $0\le \varphi\le \mathbb I_{D\setminus \overline D_0}$, where
(with the convention $\infty\cdot 0=0$)
\[
\mathbb I_{D\setminus \overline D_0}=\infty\mathbf{1}_{D_0}+\mathbf{1}_{D\setminus \overline D_0}\,.
\]
\end{enumerate}

One of the main difficulty in studies on  equations of type (\ref{eq1.3}) lies in the fact that  $b$
may vanish (the so called degenerate logistic equations).   When $b$ is  bounded away from zero,
then the term $bv_p^p$ in \eqref{eq1.3} is bounded (uniformly in $p\ge 1$) in $L^q$ norm for any $q\ge 1$, however, if we assume that $b$
is merely non-trivial, then  we are losing some control on the term $bv_p^p$, and the best  we can 
get in the limit (as $p\to\infty$) is a bounded measure. The techniques proposed and developed  by Dancer, Du and Ma in \cite{DD1,DD2,DDM2,DDM1}, and Rodrigues and   Tavares in \cite{RT}
strongly exploits 
the  properties inherent  to the Laplace operator as  locality of the operator - the evaluation of $\Delta u(x)$ depends on the values of $u$ in an arbitrarily small neighborhood of $x$ - and regularity up to the boundary of solutions to Poisson equations with smooth data and domains. Unfortunately, the fractional Laplacian does not share these properties (see \cite{RS}). Therefore, we  find ourselves forced to propose a new method of studying (\ref{eq1.3}). It  combines the techniques used in the case of the classical Laplacian with
some new technics based on the probabilistic potential theory and stochastic analysis.
Considered  method allows us  to handle  asymptotics for  \eqref{eq1.3}  with irregular data and domain, and  also  to get pointwise convergence in asymptotics results.
The last property was never investigated in the literature in the case of degenerate logistic type equations.
The method  we apply here   depends upon the knowledge that $\Delta^{\alpha}$ is symmetric, strongly Feller
and intrinsic ultracontractive, therefore the results of the paper may be easily extended to a much  broader class of operators
(including classical Laplacian). The technique  seems to be    very powerful as evidenced by the fact that
when applied to classical Laplacian, it gives stronger results than in \cite{DD1,DD2,DDM1,RT} (this is so, among others,
because we do not use the Hopf lemma, which requires high regularity of the boundary of $D$). 
Note that in the paper \cite{RT} devoted to evolution equations, the  authors assume in addition that
 $D,D_0$ are smooth domains, and in \cite{DDM1} devoted to elliptic equations the authors assume in addition that
$D,D_0$ are smooth, $b$ is continuous and $\overline D_0\subset D$.

As for  (i),  we prove that if $v$ is a  unique  solution of the  parabolic obstacle problem
\begin{equation}
\label{eq1.4}
 \left\{
\begin{array}{l}
\max\big\{ \frac{d v}{d t}-\Delta^{\alpha} v-av, v-\mathbb I_{D\setminus\overline D_0}\big\}=0,\quad\mbox{in}\,\,\, D\times (0,\infty),
\medskip\\
\,v=0,\quad\mbox{in}\,\,\, D^c\times (0,\infty),
\medskip\\
 v(0,\cdot)=\varphi,\quad\mbox{in}\,\,\, D,
 \end{array}
\right.
\end{equation}
then for all $T>0$ and $\delta\in (0,T]$,
\begin{equation}
\label{eq1.a}
\sup_{\delta\le t\le T}\|v_p(t)-v(t)\|_\infty+\int_0^T\|v_p(t)-v(t)\|_{H^{\alpha}_0(D)}\,dt \rightarrow 0\quad \mbox{as }p\rightarrow \infty.
\end{equation}
Moreover, if $\varphi\in C_0(D)$, then (\ref{eq1.a}) holds with $\delta=0$.
We then show that for  every  $a\in (\lambda_1^D, \lambda_1^{D_0})$, where
$\lambda_1^D$ (resp. $\lambda_1^{D_0}$) denotes the principal eigenvalue of the operator $-\Delta^{\alpha}$ with zero exterior Dirichlet condition on $D^c$ (resp. $D_0^c$),
there exists a (unique)  solution
$u$ of the  elliptic obstacle problem
\begin{equation}
\label{eq1.2} \left\{
\begin{array}{l}\max\big\{-\Delta^{\alpha} u-au,u- \mathbb I_{D\setminus \overline D_0}\big\}=0,\quad \mbox{in}\,\,\, D,
\medskip\\ u=0,\quad\mbox{in}\,\,\, D^c,
\medskip\\
u>0,\quad \mbox{on}\,\,\, D,
\end{array}
\right.
\end{equation}
and
\begin{equation}
\label{eq1.b}
\|v(t)-u\|_\infty+\|v(t)-u\|_{H^{\alpha}_0(D)} \rightarrow 0\quad \mbox{as }t\rightarrow \infty.
\end{equation}
As a matter of fact, in the present paper, we only show the existence of $u$. The uniqueness problem for (\ref{eq1.2}) is a separate difficult issue. It is solved
in \cite{K:arx1} (see also \cite{DDM2} for the case of the classical Laplacian).

As for problem (ii), we show that there exists a solution $u_p$ of the problem
\begin{equation}
\label{eq1.1} \left\{
\begin{array}{l}-\Delta^{\alpha} u_p= au_p-bu_p^p,\quad\mbox{in}\,\,\, D,\medskip \\
\,u=0,\quad\mbox{in}\,\,\, D^c,
\medskip \\
\,u_p>0\quad \mbox{on}\,\,\,  D
\end{array}
\right.
\end{equation}
if and only if $a\in (\lambda_1^D,\lambda_1^{D_0})$, and that for $a$ satisfying this condition,
\begin{equation}
\label{eq1.c}
\|v_p(t)-u_p\|_\infty+\|v_p(t)-u_p\|_{H^{\alpha}_0(D)} \rightarrow 0\quad \mbox{as }
t\rightarrow \infty.
\end{equation}
We next show that  for every $a\in (\lambda_1^D,\lambda_1^{D_0})$,
\begin{equation}
\label{eq1.d}
\|u_p-u\|_\infty+\|u_p-u\|_{H^{\alpha}_0(D)} \rightarrow 0\quad \mbox{as }p\rightarrow \infty,
\end{equation}
where $u$ is a solution to (\ref{eq1.2}).
The uniform convergence in (\ref{eq1.a}) and (\ref{eq1.d})  has been considered before in the literature only in the case
when $b\ge c$ for some constant $c>0$. For the proof in the general case, we combine the analytic  methods of \cite{BM1,RT}  with the Feynman-Kac representation and  some  methods of stochastic analysis and probabilistic potential theory.
In the proof of the asymptotics as $t\rightarrow\infty$, we merge  the  techniques introduced in \cite{SZ} with the probabilistic ones  introduced in \cite{K:JEE1}.
Note here that under additional regularity conditions on $b, D, D_0, \varphi$ the large time behaviour of  solutions to   (\ref{eq1.3}) and (\ref{eq1.4}) with classical Laplace operator was studied in \cite{DG,DY,FKMLG,RT}.

The method  we propose  in the paper  is built  on three pillars:  the Feynman-Kac representation of solutions to the mentioned problems
(thanks to which, among others, we achieve uniform convergences), the notion of intrinsic ultracontractivity 
which stands as a substitute of the Hopf lemma (as a byproduct, we may consider less regular domains), and the following 
result which plays a pivotal role in our proofs of the energy
estimates for solutions to (\ref{eq1.3}):
\begin{equation}
\label{eq1.1.1.1}
\lambda^D_1(-\Delta^{\alpha}+q_k)\nearrow \lambda_1^{V}\quad\mbox{as }k\rightarrow\infty.
\end{equation}
Here $V\subset D$ is a bounded {\em  Kac regular} domain (any Lipschitz domain  is Kac regular, see Proposition \ref{prop.ciea}) and  $\lambda^V(-\Delta^{\alpha}+q_k)$ is the principal eigenvalue of the operator $-\Delta^{\alpha}+q_k$ with zero exterior Dirichlet  condition on $V^c$. Furthermore, $\{q_k\}$ is an increasing sequence of bounded positive measurable functions on $D$ such that supp$[q_k]\subset D\setminus \overline V$ and
\[
\forall\,K\subset D\setminus \overline V,\, K\mbox{-compact}\qquad  \inf_{x\in K}q_k(x)\nearrow \infty.
\]
Similar result, but for classical Dirichlet Laplacian, was proved in \cite{FKMLG} under very restrictive smoothness assumptions on the domain.

It is worth mentioning that there is a rich series  of papers (see \cite{BZ,CCR} and the references therein) devoted to  nonlocal
logistic equations of the form (\ref{eq1.3}) but with $\Delta^{\alpha}$ replaced by a nonlocal operator $A$ of the form
\begin{equation}
\label{eq1.i.1}
-Au(x)=\int_DJ(x,y)(u(x)-u(y))\,dy
\end{equation}
with  some strictly positive symmetric kernel $J\in C(\overline D\times \overline D)$. By the very definition
of the fractional Laplacian representation \eqref{eq1.i.1} holds for $A=\Delta^{\alpha}$ (with the principal value integral)
but for $J(x,y)\sim |x-y|^{-d-2\alpha}$, which clearly does not belong to $C(\overline D\times \overline D)$.

\section{Preliminary results}
\label{sec2}

Throughout the paper, we assume that $d\ge2$ and $D\subset\BR^d$  is a  bounded Lipschitz domain.
We let   $m$ denote the Lebesgue measure on $\BR^d$.  We denote by $\BB(D)$ the $\sigma$-field of Borel subsets of $D$. $\BB_b(D)$ (resp. $\BB^+(D)$) is the set of real bounded (resp. positive) Borel measurable functions on $D$.  $\BB^{+}_b(D)=\BB_b(D)\cap\BB^+(D)$.

\subsection{Dirichlet fractional Laplacian and related Sobolev spaces}
\label{sec2.1}
For any $u\in C^2_b(\BR^d)$ we let
\begin{align}
\label{eq.flsid}
\nonumber
\Delta^{\alpha}u(x)&:=c_{\alpha,d}\lim_{r\rightarrow 0^+}\int_{\BR^d\setminus B(x,r)}\frac{u(y)-u(x)}{|y-x|^{d+2\alpha}}\,dy
\\&=\frac{c_{\alpha,d}}{2}\int_{\BR^d}\frac{u(x+y)+u(x-y)-2u(x)}{|y|^{d+2\alpha}}\,dy
\end{align}
with $c_{\alpha,d}=[4^{\alpha}\Gamma(\frac{d+2\alpha}{2})]/[\pi^{d/2}\Gamma(-\alpha)]$.

Let us consider the Dirichlet form  $(\EE,D(\EE))$  on $L^2(\BR^d;m)$ defined as
\[
D(\EE)=\{u\in L^2(\BR^d;m): \int_{\BR^d}|\xi|^{2\alpha}|\hat u(\xi)|^2\,d\xi<\infty\},\quad \EE(u,v)=\int_{\BR^d}|\xi|^{2\alpha}\hat u(\xi)\bar{\hat v}(\xi)\,d\xi.
\]
Here $\hat u$ stands for the Fourier transform of $u$.
By \cite[Proposition 3.4, Theorem 6.5]{DPV} (see also \cite[Lemma 3.15]{McLean}) $D(\EE)=H^{\alpha}(\BR^d)$,  and there exist
$C_1, C_2>0$ such that $C_1\|u\|_{\EE}\le\|u\|_{H^{\alpha}(\BR^d)}\le C_2\|u\|_{\EE},\, u\in D(\EE)$, where $\|u\|_\EE=\sqrt{\EE(u,u)}$, and 
\begin{equation}
\label{eq2.1.dffhs1}
H^{\alpha}(\BR^d)=\{u\in L^2(\BR^d;m): \|u\|_{H^{\alpha}(\BR^d)}:=[u]^{1/2}_{H^{\alpha}(\BR^d)}+\|u\|_{L^2(\BR^d;m)}<\infty\}
\end{equation}
with
\begin{equation}
\label{eq2.1.dffhs2}
[u]_{H^{\alpha}(\BR^d)}:= \int_{\BR^d}\int_{\BR^d}\frac{|u(x)-u(y)|^2}{|x-y|^{d+2\alpha}}\,dx\,dy.
\end{equation}
Moreover, by \cite[Proposition 3.3]{DPV}
\begin{equation}
\label{eq2.1.dffhs3}
\EE(u,v)=c_{\alpha,d}\int_{\BR^d}\int_{\BR^d}\big(u(x)-u(y)\big)\big(v(x)-v(y)\big)|x-y|^{-d-2\alpha}\,dx\,dy,\quad u,v\in D(\EE).
\end{equation}
By \cite[Sections 1.3, 1.4]{FOT} there exists a unique self-adjoint operator $(A,D(A))$ such that $D(A)\subset D(\EE)$, and 
\[
\EE(u,v)=(-Au,v),\quad u\in D(A),v\in D(\EE).
\]
From this relation and \eqref{eq2.1.dffhs3} we infer that $C^2_c(\BR^d)\subset D(A)$, and  for any $u\in C_c^2(\BR^d)$
\[
Au(x)=\Delta^{\alpha}u(x),\quad  \mbox{a.e.}\,\,\, x\in\BR^d.
\]

Let $\mbox{Cap}$ be the capacity naturally associated with the form $\EE$ (see \cite[Section 2.1]{FOT}). We say that a property  holds $\EE$-q.e. if it holds outside a set of capacity $\mbox{Cap}$ zero. We say that a function $u$
on $\BR^d$ is $\EE$-quasi-continuous if for every $\varepsilon>0$ there exists a closed set $F_\varepsilon$ such that $\mbox{Cap}(\BR^d\setminus F_\varepsilon)\le\varepsilon$ and $u_{|F_\varepsilon}$ is continuous. It is well known (see \cite[Theorem 2.1.3]{FOT}) that each $u\in D(\EE)$ has an $\EE$-quasi-continuous $m$-version, which in the sequel will be denoted by $\tilde u$.

By $\{T_t,\, t\ge 0\}$ (resp. $\{J_\beta,\, \beta\in \rho(A)\}$), we denote the semigroup (resp. resolvent)
generated by $A$.

We let   $(\EE_D,D(\EE_D))$ denote a form (called the part of $(\EE,D(\EE))$  on $D$) defined by
\[
D(\EE_D)=\{u\in D(\EE):\tilde u=0\,\, \EE\mbox{-q.e.}\,\,\mbox{on}\,\, \BR^d\setminus D\},\quad \EE_D(u,v)=\EE(u,v),\quad u,v\in D(\EE_D).
\]
By \cite[Theorem 4.4.3]{FOT}, $(\EE_D,D(\EE_D))$ is a regular Dirichlet form on $L^2(D;m)$. Therefore (see \cite[Sections 1.3, 1.4]{FOT}) there exists a unique self-adjoint negative definite operator $(A_D,D(A_D)$ on $L^2(D;m)$ such that
\[
D(A_D)\subset D(\EE_D),\quad \EE_D(u,v)=(-A_D u,v),\quad u\in D(A_D),\,v\in D(\EE_D)
\]
(here $(\cdot,\cdot)$ stands for the usual scalar product in $L^2(\BR^d;m)$). We put
\[
(\Delta^{\alpha})_{|D}:=A_D.
\]
The operator $(\Delta^{\alpha})_{|D}$ is called the  Dirichlet fractional Laplacian. 
Let $H^{\alpha}(D), [u]_{H^{\alpha}(D)}$ be defined as in \eqref{eq2.1.dffhs1}, \eqref{eq2.1.dffhs2}
but with $\BR^d$ replaced by $D$. Let
\[
H^{\alpha}_0(D):=\overline{C^\infty_c(D)}^{H^{\alpha}(D)}=\overline{C^\infty_c(D)}^{H^{\alpha}(\BR^d)}.
\]
The last equation follows from \cite[Lemma 5.1]{DPV}. On the other hand, since $(\EE_D,D(\EE_D))$
is regular, then 
\[
D(\EE_D)=\overline{C^\infty_c(D)}^{\|\cdot\|_{\EE_D}}=\overline{C^\infty_c(D)}^{\|\cdot\|_{\EE}}.
\]
The last equation follows from the definition of the form  $(\EE_D,D(\EE_D))$. By the equivalence of the norms $\|\cdot\|_{\EE}$
and $\|\cdot\|_{H^{\alpha}(\BR^d)}$, we conclude that $D(\EE_D)=H^{\alpha}_0(D)$. Observe that by \eqref{eq2.1.dffhs3}
for any $u\in D(\EE_D)$,
\[
\EE_D(u,u)=c_{\alpha,d}[u]_{H^{\alpha}(D)}+c_{\alpha,d}\int_Du^2(x)c(x)\,dx,
\]
where 
\[
c(x)=2\int_{\BR^d\setminus D}\frac{1}{|x-y|^{d+2\alpha}}\,dy,\quad x\in D.
\]
Thus, by \cite[Theorem 6.7]{DPV}, there exists $c_1>0$ such that $c_1\|u\|^2_{H^{\alpha}(D)}\le \EE_D(u,u),\, u\in D(\EE_D)$.
On the other hand for any $u\in C_c^\infty(D)$, by the equivalence of the norms $\|\cdot\|_{\EE}$
and $\|\cdot\|_{H^{\alpha}(\BR^d)}$, and  \cite[Theorem 5.4]{DPV}, we have
\[
\EE_D(u,u)=\EE(u,u)\le C_3\|u\|^2_{H^{\alpha}(\BR^d)}\le C_4\|u\|^2_{H^{\alpha}(D)}.
\]
Consequently, there exist $c_1, c_2>0$ such that
\begin{equation}
\label{eq2.1.equiv1}
c_1\|u\|_{\EE_D}\le \|u\|_{H^{\alpha}(D)}\le c_2\|u\|_{\EE_D},\quad u\in D(\EE_D)=H^{\alpha}_0(D).
\end{equation}
Recall here (see e.g. \cite[Corollary 7.2]{DPV}) that for any $p\in [1,2d/(d-2\alpha))$,
\begin{equation}
\label{eq2.1.equiv15r}
 H^{\alpha}(D) \subset L^p(D;m),
\end{equation}
and the embedding is compact.

As in the case of the form $(\EE,D(\EE))$, one can define  capacity $\mbox{Cap}_D$ and the notions of $\EE_D$-exceptional sets and  $\EE_D$-quasi-continuity.  We will drop $\EE_D$ in the notation if it will be clear from the context which Dirichlet form is considered. Note, however, that on the set $D$ both capacities, i.e. $\mbox{Cap}_D$ and $\mbox{Cap}$ are equivalent,  and the notions of  $\EE_D$-quasi-continuity and  $\EE$-quasi-continuity agree (see \cite[Theorem 4.4.3]{FOT}).

\subsection{Probabilistic potential theory}
\label{sec2.2}

Let  $\mathbb X= ((X_t)_{t\ge0}, (P_x)_{x\in\BR^d},(\FF_t)_{t\ge0})$ be a rotation invariant $\alpha$-stable  L\'evy process
associated with $(\EE,D(\EE))$ in the sense that  for any positive Borel function $f\in L^2(E;m)$,
\[
P_tf(x):=\mathbb E_xf(X_t)=T_t f(x)\quad \mbox{a.e. }x\in\BR^d,
\]
where $\mathbb E_x$ denotes the expectation  with respect to the measure $P_x$. It is well known (see e.g. \cite[Proposition I.2.5]{Bertoin} and \cite[Exercise 4, page 39]{Bertoin}) that such a process is doubly Feller, i.e. it is strongly Feller: $P_t(\BB_b(\BR^d))\subset C_b(\BR^d),\, t>0$, and it is Fellerian: $P_t(C_0(\BR^d))\subset C_0(\BR^d),\, t>0$. Here $C_b(\BR^d)$ is the set of bounded continuous functions on $\BR$, and $C_0(\BR^d)$ is the set of continuous functions on $\BR^d$ vanishing at infinity. We denote by $\mathbb X^D$ the process $\mathbb X$ killed upon exiting  $D$. It is known that $\BX^D$ is associated with the form $(\EE_D,D(\EE_D))$ (see \cite[Theorem 4.4.2]{FOT}). This means that for any positive Borel  $f\in L^2(D;m)$,
\begin{equation}
\label{eq.ppt1}
P^D_tf(x):=\mathbb E_x [f(X_t)\mathbf{1}_{t<\tau_D}]=T^D_t f(x),\quad \mbox{a.e.}\,\,\, x\in D,
\end{equation}
\begin{equation}
\label{eq.ppt12}
R^D_\beta f(x):=\mathbb E_x\int_0^{\tau_D} e^{-\beta r}f(X_r)\,dr=J^D_\beta f(x),\quad \mbox{a.e.}\,\,\, x\in D,
\end{equation}
where
\[
\tau_D=\inf\{t>0: X_t \in\BR^d\setminus D\}.
\]
Here $(T^D_t)$ is a Markov semigroup generated by $(\Delta^{\alpha})_{|D}$ on $L^2(D;m)$, and $(J^D_\beta)$ is its resolvent (note that $[0,\infty)$ is included in the resolvent set of $(\Delta^{\alpha})_{|D}$).
We denote by $G_{D,\beta}$ the   $\beta$-Green function for the operator $-(\Delta^{\alpha})_{|D}$, and by $p_D$  its transition function (see \cite[Exercise 4.2.1, Lemma 4.2.4]{FOT} for details). By the definition, for any $f\in \BB^+(D)$,
\begin{equation}
\label{eq.dfl1}
P^D_tf=\int_D f(y)p_D(t,\cdot,y)\,dy,\quad R^D_\beta f=\int_DG_{D,\beta}(\cdot,y)f(y)\,dy
\quad  x\in D.
\end{equation}
We write $G_D= G_{D,0}$. For a positive Borel measure $\mu$ on $D$, we set
\[
P^D_t\mu(x)=\int_D p_D(t,x,y)\,\mu(dy),\qquad R^D_\beta \mu(x)= \int_D G_{D,\beta}(x,y)\,\mu(dy)
\]
and $R^D=R^D_0$. By \cite[Theorem 4.2.3]{FOT}, for any $f\in \BB(D)\cap L^2(D;m)$ we have 
\begin{equation}
\label{eq.dfl101}
P_t^Df=\widetilde{T^D_t f},\quad \mbox{q.e.}\,\,\, x\in D,\, t>0,\qquad  R^D_\beta f=\widetilde{J^D_\beta f},\quad\mbox{q.e.}\,\,\, x\in D,\, \beta>0.
\end{equation}
It is well known (see e.g. \cite[Lemma 2.1]{Grzywny}), that
\begin{equation}
\label{eq.estgfu}
\sup_{x\in D}R^D1(x)<\infty
\end{equation}
It is also well known (see e.g. \cite[Section 4.2]{Kwasnicki}) that there exists $c>0$ (depending on $\alpha, d$) such that 
\begin{equation}
\label{eq.estgfu1}
G_D(x,y)\le \frac{c}{|x-y|^{d-2\alpha}},\quad x,y\in D.
\end{equation}
This in turn implies that 
\begin{equation}
\label{eq.estgfu2}
\sup_{x\in D}\|G_D(x,\cdot)\|_{L^p(D;m)}<\infty,\quad \mbox{for}\,\,\, p\in [1,d/(d-2\alpha)).
\end{equation}

We say that a Borel measure $\mu$ on $D$ is $\EE_D$-smooth if $|\mu| \ll\mbox{Cap}_D$ (i.e. $\mu$ charges no set of capacity $\mbox{Cap}_D$ zero)
and there exists an increasing sequence $\{F_n\}$ of closed subsets of $D$
such that $|\mu|(F_n)<\infty$, $n\ge 1$, and $\mbox{Cap}_D(K\setminus F_n)\rightarrow 0$ as $n\rightarrow\infty$ for every compact set $K\subset D$.

We let $\MM_{0,b}(D)$  be the space of all bounded $\EE_D$-smooth measures on $D$. We say that  an $\EE_D$-smooth measure $\mu$ belongs to the class $S_0(D)$ if there exists $c>0$ such that
\begin{equation}
\label{eq.ppt2}
\int_D|\tilde u|\,d|\mu|\le c\sqrt{\EE_D(u,u)},\quad u\in D(\EE_D).
\end{equation}
In the light of \eqref{eq2.1.equiv1}, $S_0(D)\subset H^{-\alpha}(D)$. 
We denote by  $\|\mu\|_{H^{-\alpha}(D)}$ the smallest  $c\ge 0$ such that \eqref{eq.ppt2} holds.
It is well known that there is a one-to-one correspondence  (Revuz duality) between positive $\EE_D$-smooth measures and positive continuous additive functionals (PCAFs for short)
of $\mathbb X^D$ (see \cite[Theorem 5.1.4]{FOT}).
By \cite[Theorem 5.1.3]{FOT}, if $A^\mu$ is the unique PCAF associated with a positive smooth measure $\mu$, then for every positive Borel function $f$ on $D$,
\begin{equation}
\label{eq.ppt3}
\mathbb E_x\int_0^{\tau_D} e^{-\beta r}f(X_r)\,dA^\mu_r=R^D_\beta(f\cdot \mu)(x)\quad \mbox{q.e. }x\in D,
\end{equation}
where $f\cdot \mu$ is a positive Borel measure  such that $d(f\cdot\mu)/d\mu=f$.

In what follows, if there is no ambiguity, in the notation we drop the prefix $\EE_D$.

\begin{lemma}
\label{lm2.4.1}
Assume that  $u\in\BB(D)$ and $\mu\in\mathcal M_{0,b}(D)$. Then
\[
u(x)=R^D\mu(x)
\]
for q.e. $x\in D$ if and only if there exists a process $M$ with $M_0=0$ such that $M$ is a uniformly integrable martingale on $[0,\tau_D]$ under the measure $P_x$ for q.e. $x\in \BR^d$, and 
\[
u(X_t)=\int_t^{\tau_D}\,dA^\mu_r-\int_t^{\tau_D}\,dM_r,\quad t\in [0,\tau_D],\quad P_x\mbox{-a.s.},\quad \mbox{q.e.}\,\,\, x\in D.
\]
\end{lemma}
\begin{proof}
See \cite[Theorem 4.7]{KR:JFA}.
\end{proof}

\subsection{Regular domains}

We say that a bounded domain $D\subset \BR^d$ is {\em Dirichlet regular} if for every $x\in \BR^d\setminus D$,
\[
P_x(\tau_D>0)=0.
\]
To see that any Lipschitz domain is Dirichlet regular, we first recall the notion of the base of a set $A\subset \BB(\BR^d)$ (see \cite[Section VI.4]{BH}).
Let
\[
b(A)=\{x\in\BR^d: P_x(\sigma_A=0)=1\},\quad \sigma_A=\inf\{t>0: X_t\in A\}.
\]
We have  that  for any $A,B\subset \BB(\BR^d)$,
\begin{equation}
\label{eq4.2}
b(A\cup B)=b(A)\cup b(B).
\end{equation}
Indeed, the inclusion $b(A)\cup b(B)\subset b(A\cup B)$
is clear (since $\sigma_{A\cup B}\le \sigma_A\wedge \sigma_B$). For the reverse inclusion, observe that $\sigma_{A\cup B}=0$
implies that $\sigma_A=0$ or $\sigma_B=0$. Therefore, if $x\in b(A\cup B)$, then $P_x(\sigma_A=0)>0$ or $P_x(\sigma_B=0)>0$.
By  Blumenthal's zero–one law this implies that $P_x(\sigma_A=0)=1$ or $P_x(\sigma_B=0)=1$. Thus, $x\in b(A)\cup b(B)$.
 
The following reasoning is taken from \cite[Example VI.4.7.4, page 276]{BH}. Let $x\in \partial D$, and $C_x$ be an open exterior cone at $x$. There exist rotations $\Phi_1,\dots \Phi_m:\BR^d\rightarrow \BR^d$
 such that $\overline D^c \setminus\{x\}=\bigcup_{i=1}^m\Phi_i(\overline D^c\cap C_x)$. Clearly, $x\in b(\overline D^c\setminus\{x\})$, so by (\ref{eq4.2}) there exists $i_0\in \{1,\dots,m\}$ such that $x\in b(\Phi_{i_0}(\overline D^c\cap C_x))$. By the rotation invariance of $\mathbb X$, $x\in b(\overline D^c\cap C_x)$.  Thus, $P_x(\sigma_{\overline D^c\cap C_x}=0)=1$.
Consequently, 
\begin{equation}
\label{eq.swo}
P_x(\sigma_{\overline D^c}=0)=1,\quad x\in\partial D.
\end{equation}
 Hence,   $P_x(\tau_{ D}>0)=0,\, x\in D^c$. We see that exterior cone condition is sufficient
 for Dirichlet regularity of $D$. In particular Lipschitz domains are Dirichlet regular.

We say that  a domain $D$ is {\em Kac regular} if
\begin{equation}
\label{eq.rd1}
P_x(\tau_D=\tau_{\overline D})=1\quad x\in D.
\end{equation}

\begin{proposition}
\label{prop.ciea}
If $D$ is a bounded Lipschitz domain, then $D$ is Kac regular.
\end{proposition}
\begin{proof}
Let $x\in D$. Clearly, $\tau_D\le\tau_{\overline D}$. By the strong Markov property
\begin{equation}
\label{eq.swo122}
\mathbb E_x\big(\mathbf{1}_{\{\tau_{\overline D}=\tau_D\}}\big|\FF_{\tau_D}\big)=P_{X_{\tau_D}}(\tau_{\overline D}=0),\quad P_x\mbox{-a.s.}
\end{equation}
By (\ref{eq.swo}), $P_y(\tau_{\overline D}=0)=1,\, y\in\partial D$. If $y\in \overline D^c$, then obviously $P_y(\tau_{\overline D}=0)=1$.
Thus,
\begin{align*}
\mathbb E_x P_{X_{\tau_D}}(\tau_{\overline D}=0)&=\mathbb E_x\mathbf{1}_{\{X_{\tau_D}\in \partial D\}}P_{X_{\tau_D}}(\tau_{\overline D}=0)+
\mathbb E_x\mathbf{1}_{\{X_{\tau_D}\in \overline D^c\}} P_{X_{\tau_D}}(\tau_{\overline D}=0)\\&= P_x(X_{\tau_D}\in \partial D)+P_x(X_{\tau_D}\in \overline D^c)=1.
\end{align*}
From this and (\ref{eq.swo122}), we get (\ref{eq.rd1}).
\end{proof}

\subsection{Eigenfunctions and intrinsic ultracontractivity}

In what follows, we denote by $\lambda_1(B)$ the first eigenvalue (whenever it exists) of a given operator $B$. To simplify notation, we also set
\[
\lambda_1^D:= \lambda_1(-(\Delta^{\alpha})_{|D}),\quad \lambda_1^D[q]:= \lambda_1(-(\Delta^{\alpha})_{|D}+q),
\]
where $q:D\to \BR^+$ is a positive potential.
It is  well known (see e.g. \cite{BG0}) that
\begin{equation}
\label{eq.tg.1}
0<p_D(t,x,y)\le \hat c \min\Big\{\frac{1}{t^{d/2\alpha}}, \frac{t}{|x-y|^{d+2\alpha}}\Big\},\quad x,y\in D,\quad t>0.
\end{equation}
Therefore, by Jentzsch's theorem (see e.g.  \cite[Theorem V.6.6, page 338]{Schaefer} or \cite[Lemma 6.4.5]{FOT})  for any  bounded open domain $D\subset \BR^d$  and a positive $q\in \BB_b(D)$ there exists a  unique strictly positive continuous  eigenfunction $\psi$ associated with the eigenvalue $\lambda^D_1[q]>0$ such that $\|\psi\|_{L^2(D;m)}=1$. We call $\psi$ the {\em principal eigenfunction}  for the operator $-(\Delta^{\alpha})_{|D}+q$, and  $\lambda_1^D[q]$ the {\em principal eigenvalue}. Moreover, if $D$ is Dirichlet regular, then $\psi\in C_0(D)$.  
We  denote by $\varphi_1^D$ the principal  eigenfunction for $-(\Delta^{\alpha})_{|D}$. From (\ref{eq.tg.1})  it also follows   that  for any $q\in [1,(d+2\alpha)/d)$ and  $T>0$,
\begin{equation}
\label{eq.tg.2}
\sup_{x\in D}\int_0^T\|p_D(t,x,\cdot)\|_{L^q(D;m)}\,dt<\infty.
\end{equation}

Recall that a symmetric Markov semigroup $(Q_t)$ on $L^2(D;m)$ is said to be  {\em ultracontractive} if 
for any $t>0$, $Q_t: L^2(D;m)\to L^\infty(D;m)$ is bounded. In this case, there exists a transition function $q(\cdot,\cdot,\cdot)$
such that for any $f\in\BB^+(D)$,
\[
Q_tf(x)=\int_Df(y)q(t,x,y)\,m(dy),\quad x\in D,\, t>0,
\]
and for any $t>0$ there exists $c_t>0$ such that $q(t,x,y)\le c_t,\, x,y\in D$.
By \cite{Grzywny},  the semigroup $(P^D_t)_{t>0}$ is {\em intrinsically ultracontractive} (the notion introduced in  \cite{DS}), i.e. for any  $t>0$, a Markov semigroup $(Q^D_t)$ defined as
\[
Q^D_tf(x)=e^{t\lambda_1^D}\frac{P^D_t(\varphi_1^Df)(x)}{\varphi_1^D(x)},\quad x\in D
\]
is ultracontractive on $L^2(D,\nu)$ with $\nu:= (\varphi_1^D)^2\cdot m$. Observe that the transition density $q_D(\cdot,\cdot,\cdot)$
for $(Q^D_t)$ admits the following formula
\[
q_D(t,x,y)=e^{t\lambda_1^D}\frac{p_D(t,x,y)}{\varphi^D_1(x)\varphi_1^D(y)},\quad x,y\in D,\, t>0.
\]
Therefore, for any $t>0$, there exists $\beta_t>0$ such that 
\begin{equation}
\label{eq.iuc12u}
p_D(t,x,y)\le \beta_t e^{-t\lambda_1^D } \varphi^D_1(x)\varphi_1^D(y),\quad x,y\in D.
\end{equation}
By \cite[Lemma 2.1.2]{Davies}, $\beta_t$ is non-increasing as  $t\rightarrow \infty$.
Moreover, by \cite[Theorem 3.2]{DS} for any $t>0$, there exists $\alpha_t>0$ such that
\begin{equation}
\label{eq.iuc12l}
\alpha_t e^{-t\lambda_1^D }\varphi_1^D(x)\varphi_1^D(y)\le p_D(t,x,y),\quad x,y\in D.
\end{equation}

\section{Eigenfunctions and eigenvalues}
\label{sec3}

From now on, unless it is stated otherwise,  we assume that (H1), (H2) (see Introduction) are satisfied.
Recall that by the Rayleigh-Ritz variational formula (see e.g. \cite[Theorem 4.5.1]{Davies}),
\[
\lambda^D_1[q]=\inf\{\EE_D(u,u)+(qu,u)_{L^2(D;m)}: u\in D(\EE_D),\, \|u\|_{L^2(D;m)}=1\},
\]
for any bounded open domain $D\subset\BR^d$, and for any $q\in\BB_b^+(D)$.
In particular, for open bounded domains $D_0\subset D\subset \BR^d$, we have
\begin{equation}
\label{eq3.rrvf12}
\lambda^D_1[q]\le \lambda_1^{D_0}[q].
\end{equation}

\begin{theorem}
\label{th3.1}
Let $c\in\BB^+_b(D)$, and $\{q_k\}\subset\BB^+_b(D)$ be a sequence
such that \mbox{\rm supp}$[q_k]\subset D\setminus \overline D_0$, $k\ge 1$, and for every compact $K\subset D\setminus \overline D_0$, $\inf_{x\in K}q_k(x)\nearrow \infty$ as $k\rightarrow\infty$. Let $\psi_k^c$ (resp. $\psi^c$) be the  principal eigenfunction for  $-(\Delta^{\alpha})_{|D}+c+q_k$ (resp. $-(\Delta^{\alpha})_{|D}+c$). Then
\[
\lambda^D_1[c+q_k]\nearrow \lambda^{D_0}_1[c],
\]
and there exists a subsequence $\{k_n\}$ such that
\[
\psi_{k_n}^c(x)\rightarrow \psi^c(x)\quad \mbox{for q.e. }x\in D_0.
\]
\end{theorem}
\begin{proof}
By the definition of the principal eigenvalue,
\begin{equation}
\label{eq3.6}
\EE_D(\psi^c_k,\psi^c_k)+(c\psi^c_k,\psi^c_k)+(q_k\psi^c_k,\psi^c_k)=\lambda^D_1[c+q_k].
\end{equation}
By \cite[Theorem 6.1.1]{FOT},
\begin{equation}
\label{eq3.7}
\psi^c_k(x)=\lambda^D_1[c+q_k]\mathbb E_x\int_0^{\tau_D}e^{-\int_0^t(c+q_k)(X_r)\,dr}\psi^c_k(X_t)\,dt,\quad x\in D.
\end{equation}
By (\ref{eq3.rrvf12}),
\[
\lambda^D_1[c+q_k]\le \lambda^{D_0}_1[c+q_k]=\lambda^{D_0}_1[c].
\]
From  this and (\ref{eq3.6}), we get
\begin{equation}
\label{eq3.9}
\sup_{k\ge 1}\EE_D(\psi^c_k,\psi^c_k)\le \sup_{k\ge 1}\lambda^D_1[c+q_k]
=:\lambda<\infty.
\end{equation}
By \eqref{eq2.1.equiv15r},
there exists a subsequence $\{n_k\}$ such that
$\psi^c_{k_n}\rightarrow \psi $ in $L^2(D;m)$ for some $\psi\in
L^2(D;m)$ such that $\|\psi\|_{L^2(D;m)}=1$. Let $\nu\in S_0(D)$ (cf. \eqref{eq.ppt2})
be a positive measure such that $R^D\nu$ is bounded q.e by a
constant. Then
\[
\int_D\mathbb E_x\int_0^{\tau_D}|\psi^c_k(X_r)-\psi(X_r)|\,dr\,d\nu=(|\psi^c_k-\psi|,R^D\nu)\le c\|\psi^c_k-\psi\|_{L^2(D;m)}.
\]
From this and a standard reasoning (see the reasoning following  \cite[(5.2.22)]{FOT}), we infer that, up to subsequence,
\begin{equation}
\label{eq3.cutr5}
\mathbb E_x \int_0^{\tau_D}|\psi^c_k(X_r)-\psi(X_r)|\,dr\to 0,\quad\mbox{q.e.}\,\,\, x\in D.
\end{equation}
Applying this convergence  and (\ref{eq3.7}) we deduce that, up to a subsequence, $\{\psi^c_{k_n}\}$ is convergent q.e.
Set $\tilde\psi(x)=\limsup_{n\rightarrow \infty}\psi^c_{k_n}(x)$, $x\in D$. Observe that by the assumptions on the sequence $\{q_k\}$,
\[
\int_0^tq_k(X_r)\rightarrow \infty\mathbf{1}_{\{\tau_{ \overline D_0}<t\}},\quad t\in [0,\tau_D],
\]
with the convention that $0\cdot \infty=0$. Hence
\[
e^{-\int_0^tq_k(X_r)\,dr}\rightarrow \mathbf{1}_{[0, \tau_{ \overline D_0}]}(t),\quad t\in [0,\tau_D].
\]
From this, (\ref{eq3.7}), \eqref{eq3.cutr5} and q.e. convergence of $\psi^c_{k_n}$ we conclude 
\[
\tilde \psi(x)=\lambda \mathbb E_x\int_0^{\tau_{\overline D_0}}e^{-\int_0^t c(X_r)\,dr}\tilde\psi(X_t)\,dt,\quad\mbox{q.e.}\,\,\, x\in D_0.
\]
By Proposition \ref{prop.ciea}, 
\[
\tilde \psi(x)=\lambda \mathbb E_x\int_0^{\tau_{D_0}}e^{-\int_0^t c(X_r)\,dr}\tilde\psi(X_t)\,dt,\quad\mbox{q.e.}\,\,\, x\in D_0.
\]
By the above formula, we get, in particular, that $\tilde \psi$ is strictly positive and quasi-continuous
on $D_0$. Since $\tilde\psi=\psi$ a.e., $\tilde \psi\in D(\EE_D)$. Hence, by \cite[Theorem 6.1.1]{FOT}, $\tilde \psi$ is a strictly positive solution to
$-(\Delta^{\alpha})_{|D_0}u+cu=\lambda u$.
Therefore $\tilde\psi=\psi^c$ q.e.  and $\lambda=\lambda^{D_0}_1[c]$.
\end{proof}

\begin{lemma}
\label{lm3.1} Let $q_1, q_2\in\BB^+_b(D)$. If $q_1\le q_2$  a.e. and
$m(\{q_1<q_2\})>0$, then
\[
\lambda^{D}_1[q_1]<\lambda^{D}_1[q_2].
\]
\end{lemma}
\begin{proof}
Let $\psi_1, \psi_2$ be the principal eigenfunctions for
$\lambda^{D}_1[q_1]$ and $\lambda^{D}_1[q_2]$, respectively. It
is clear that $\lambda^{D}_1[q_1]\le \lambda^{D}_1[q_2]$.
Suppose that  $\lambda:=\lambda^{D}_1[q_1]=\lambda^{D}_1[q_2]$.
Then
\[
\EE_D(\psi_1,\psi_2)+(q_1\psi_1,\psi_2)=\lambda (\psi_1,\psi_2),\qquad
\EE_D(\psi_2,\psi_1)+(q_2\psi_2,\psi_1)=\lambda (\psi_2,\psi_1).
\]
Hence $(q_2-q_1,\psi_1\psi_2)=0$, which contradicts  the fact that
$m(\{q_1<q_2\})>0$.
\end{proof}

\begin{lemma}
\label{lm3.2}
We have $\lambda^D_1<\lambda^{D_0}_1$  (cf. {\rm(H1)}).
\end{lemma}
\begin{proof}
Let $\{q_k\}\subset\BB^+_b(D)$  be a sequence of functions such
that supp$[q_k]\subset D\setminus \overline D_0$ and for every
compact $K\subset D\setminus \overline D_0$, $\inf_{x\in
K}q_k(x)\nearrow \infty$. By Theorem \ref{th3.1},
$\lambda_1^D[q_k]\nearrow \lambda_1^{D_0}$ as
$k\rightarrow\infty$, whereas by Lemma \ref{lm3.1},
$\lambda_1^D< \lambda_1^D[q_k]$, $k\ge1$, which proves the
lemma.
\end{proof}

\section{Existence result for semilinear elliptic equations}
\label{sec4}

We recall that we assume that (H1), (H2) are satisfied and $D$ is a bounded Lipschitz domain.

Let  $f:D\times \mathbb R\rightarrow \mathbb R$ be a continuous function which is bounded on bounded subsets of $D\times \BR$.  We consider the following  problem:
 \begin{equation}
\label{eq3.1}
-\Delta^{\alpha} u= f(\cdot,u),\quad {\rm in}\,\,\, D,\qquad u=0,\quad{\rm in}\,\,\, \BR^d\setminus D.
\end{equation}

\begin{definition}
\label{df3.1}
We say that a bounded function $u$ on $D$ is a solution to (\ref{eq3.1}) if
for q.e. $x\in D$,
\begin{equation}
\label{eq.er1}
u(x)=R^Df(\cdot,u)(x).
\end{equation}
\end{definition}

\begin{remark}
In the present section, we choose  as  basis  the above  integral form  definition of a solutions to PDE \eqref{eq3.1}
since it is suitable  for the method of sub and supersolutions we apply below. However, the above definition
is related to the following more familiar  definitions.
Let $u$ be a bounded function on $\BR^d$.
\label{rem4.2}
\begin{itemize}
\item[a)]
By (\ref{eq.dfl101}),  $u$ is a  solution to (\ref{eq3.1}) if and only if $u$ is a {\em weak solution} to (\ref{eq3.1}), i.e.  $u\in H^{\alpha}_0(D)$
and
\[
\EE_D(u,v)=(f(\cdot,u),v),\quad v\in H^{\alpha}_0(D).
\]
\item[b)] By \eqref{eq.ppt12}, $u$ is a  solution to (\ref{eq3.1}) if and only if $u\in D(-(\Delta^{\alpha})_{|D})$, and 
\[
-(\Delta^{\alpha})_{|D}u=f(\cdot,u)
\] 
in $L^2(D;m)$.
\item[c)] By  \cite[Theorem 7.1]{Kwasnicki}, if $u$ is a continuous  solution to (\ref{eq3.1}), then
\[
-\Delta^{\alpha}u(x)=f(x,u(x)),\quad x\in D,\quad\quad u(x)=0,\quad x\in \BR^d\setminus D,
\]
i.e. the limit in \eqref{eq.flsid} exists and the above equation holds for any $x\in \BR^d\setminus D$.
Conversely, assume that $u\in\BB_b(\BR^d)\cap C(D)$  and satisfies  the above equations, i.e. $u$
is a {\em pointwise solution} to \eqref{eq3.1}. Then $u$ is a solution to \eqref{eq3.1}
(see e.g. \cite[Theorem 7.2]{Kwasnicki}).
\end{itemize}
\end{remark}

\begin{definition}
\label{def3.2}
We say that a bounded function $u$ on $D$ is a {\em  supersolution} (resp. {\em subsolution}) to (\ref{eq3.1}) if there exists a positive (resp. negative) measure $\mu\in\MM_{0}(D)$ such that   for q.e. $x\in D$,
\[
u(x)=R^Df(\cdot,u)(x)+R^D\mu(x).
\]
\end{definition}

\begin{proposition}
\label{prop3.3}
 Let $\underline u$ (resp. $\bar u$) be a  subsolution (resp. supersolution) to {\rm{(\ref{eq3.1})}} and $\underline u\le \bar u$. Then there exists a  solution $u$ to {\rm{(\ref{eq3.1})}}
such that $\underline u\le u\le \bar u $.
\end{proposition}
\begin{proof}
{\bf Step 1.}
Define
\[
\hat f(x,y)= f(x, (\bar u(x)\wedge y)\vee \underline u(x)),\quad x\in D,\, y\in\BR.
\]
We shall show that if  $\hat u$ is a solution  to \eqref{eq3.1}, with $f$ replaced by $\hat f$,
then $\underline u\le \hat u\le \bar u$. By the definition of a supersolution, there exists a positive $\mu\in \MM_{0,b}(D)$ such that
\[
\bar u= R^Df(\cdot,\bar u)+R^D\mu\quad \mbox{q.e.}
\]
By Lemma \ref{lm2.4.1}, there exist  martingales $\hat M, \bar M$
such that for
q.e. $x\in D$,
\[
\hat u(X_t)=\int_t^{\tau_D}\hat f(X_r,\hat u(X_r))\,dr-\int_t^{\tau_D}\,d\hat M_r,\quad t\in [0,\tau_D],\quad P_x\mbox{-a.s.,}
\]
\[
\bar u(X_t)=\int_t^{\tau_D} f(X_r,\bar u(X_r))\,dr+\int_t^{\tau_D}\,dA^\mu_r-\int_t^{\tau_D}\,d\bar M_r,\quad t\in [0,\tau_D],\quad P_x\mbox{-a.s.}
\]
By the Tanaka-Meyer formula (see, e.g., \cite[IV.Theorem 70]{Protter}),
\begin{align*}
(\hat u-\bar u)^+(x)&\le \mathbb E_x\int_0^{\tau_D}\mathbf{1}_{\{\hat u>\bar u\}}(X_r)(\hat f(X_r,\hat u(X_r))-f(X_r,\bar  u(X_r)))\,dr\\&\quad -\mathbb E_x \int_0^{\tau_D}\mathbf{1}_{\{\hat u>\bar u\}}(X_r)\,dA^\mu_r.
\end{align*}
By the definition of $\hat f$ and positivity of $\mu$ we get $(\hat u-\bar u)^+=0$. A similar argument shows that $(\underline u-\hat u)^+=0$. Thus $\underline u\le \hat u\le \bar u$ as claimed.

{\bf Step 2.} Define $\Phi$ by
\[
\Phi: L^2(D;m)\rightarrow L^2(D;m),\quad \Phi(u)= R^D\hat f(\cdot, u).
\]
Since $\hat f$ is bounded, the operator $\Phi$ is well defined (cf. \eqref{eq.estgfu}).  From continuity of $f$ it follows at once that  $\Phi$
is continuous.  Let $\{u_n\}\subset L^2(D;m)$. By \eqref{eq2.1.equiv15r}, \eqref{eq.dfl101}, there exists a subsequence $\{n_k\}$
such that $R^D\hat f(\cdot, u_{n_k})$ is convergent a.e. Applying \eqref{eq.estgfu} and the dominated convergence theorem shows the convergence of $\{R^D\hat f(\cdot, u_{n})\}$ in $L^2(D;m)$. Thus  $\Phi$ is compact. Therefore, by  Schauder's fixed point theorem, there exists $u\in L^2(D;m)$ such that $u=R^D\hat f(\cdot,u)$ a.e. Of course, we may choose an $m$-version  $\hat u$ of $u$ such that $\hat u=R^D\hat f(\cdot,\hat u)$. By  Step 1, $\underline u\le \hat u\le\bar u$, so  $\hat u=R^Df(\cdot,\hat u)$.
\end{proof}

\begin{proposition}
\label{prop3.2} Assume that $a>0$. Then there exists at most one
strictly positive  solution to {\rm{(\ref{eq1.1})}}.
\end{proposition}
\begin{proof}
Let $u_1, u_2$ be strictly positive  solutions to (\ref{eq1.1}). It
is an elementary check that $u_1+u_2$ is a  supersolution to
(\ref{eq1.1}). It is also well known (see e.g. \cite[Proposition 3.7]{K:arx1}) that
$u_1\vee u_2$ is a subsolution to (\ref{eq1.1}). This, when
combined with Proposition \ref{prop3.3}, shows that without loss of
generality we may assume that $u_1\le u_2$. Striving for a contradiction, suppose that
$m(\{u_1<u_2\})>0$. By the Feynman-Kac formula (see e.g. \cite[Theorem 6.1.1]{FOT}), for every $x\in
D$,
\begin{align*}
(u_2-u_1)(x)&=a \mathbb E_x\int_0^{\tau_D}
e^{-\int_0^t(b(u_2)^{p-1}-b(u_1)^{p-1})(X_r)\,dr}(u_2-u_1)(X_t)\,dt\\&
\ge
a \mathbb E_x\int_0^{\tau_D}
e^{-t\|b\|_\infty\|u_2\|^{p-1}}(u_2-u_1)(X_t)\,dt=aR^D_{\beta}(u_2-u_2)(x),
\end{align*}
with $\beta=\|b\|_\infty\|u_2\|^{p-1}$.
It follows from this and \eqref{eq.tg.1}  that $u_1(x)<u_2(x),\, x\in D$.  In
particular, since $b$ is nontrivial and positive,
$m(\{b(u_1)^{p-1}<b(u_2)^{p-1}\})>0$. Hence, by Lemma \ref{lm3.1},
\[
a=\lambda_1^D[b(u_1)^{p-1}]<\lambda_1^D[b(u_2)^{p-1}]=a,
\]
which is a  contradiction.
\end{proof}

\begin{theorem}
\label{th3.2} There exists a solution to {\rm{(\ref{eq1.1})}}  if
and only if $\lambda_1^D<a<\lambda_1^{D_0}$.
\end{theorem}
\begin{proof}
Suppose that there exists a solution $u$ to (\ref{eq1.1}).  Then,
by Lemma \ref{lm3.1},
\[
a=\lambda_1^D[bu^{p-1}]>\lambda_1^D.
\]
Moreover, by Lemma \ref{lm3.2} and (H1),
\[
a=\lambda_1^D[bu^{p-1}]<\lambda_1^{D_0}[bu^{p-1}]=\lambda_1^{D_0}.
\]
Now, assume that $\lambda_1^D<a<\lambda_1^{D_0}$.  By Theorem
\ref{th3.1} and the fact that $a<\lambda_1^{D_0}$, there exists a
positive  $\eta\in C_c^\infty(D)$ with supp$[\eta]\subset
D\setminus\overline D_0$ such that $\lambda_1^D[\eta]\ge a$. By
(H1), there exists a  positive function  $v\in C^\infty_c(D)$
 such that $bv^{p-1}\ge \eta$. We thus have that $\lambda_1^D[bv^{p-1}]\ge \lambda_1^D[\eta]\ge a$.
Let $\psi$ be the principal eigenfunction for
$\lambda_1^D[bv^{p-1}]$, and let $c>0$ be such that $c\psi\ge
v$. Then
\[
-(\Delta^{\alpha})_{|D}(c\psi)=a(c\psi)-b(c\psi)^p
+\big(\lambda_1^D[bv^{p-1}]-a\big)c\psi+c\psi b\big((c\psi)^{p-1}-v^{p-1}\big).
\]
Therefore $c\psi$ is a supersolution to (\ref{eq1.1}). It is clear
that for a sufficiently small $\varepsilon>0$, $\varepsilon
\varphi_1^D$ is a subsolution to (\ref{eq1.1}). Moreover, by the
Feynman-Kac formula (see e.g. \cite[Theorem 6.1.1]{FOT}) and ultracontractivity of $p_D$, for every
$x\in D$ we have
\begin{align*}
\psi(x)&=e^{t\cdot \lambda_1^D[bv^{p-1}]}
\mathbb E_xe^{-\int_0^tbv^{p-1}(X_r)\,dr}\mathbf{1}_{\{t<\tau_D\}}\psi(X_t)\\&
\ge e^{t(\lambda_1^D[bv^{p-1}]-\|bv^{p-1}\|_\infty)}
\mathbb E_x\mathbf{1}_{\{t<\tau_D\}}\psi(X_t)\\&
=e^{t(\lambda_1^D[bv^{p-1}]-\|bv^{p-1}\|_\infty)}
\int_Dp_D(t,x,y)\psi(y)\,dy\\&\ge
c_t\int_D\varphi_1^D(x)\varphi_1^D(y) \psi(y)\,dy \ge \bar c_t
\varphi_1^D(x).
\end{align*}
Thus, for a sufficiently small $\varepsilon>0$, $\varepsilon
\varphi_1^D\le c\psi$. Therefore, by Proposition \ref{prop3.3},
there exists a solution to (\ref{eq1.1}).
\end{proof}

\section{Obstacle problem and asymptotics as $p\rightarrow \infty$ for elliptic equations}

In this section, we provide three equivalent formulations of the obstacle problem (\ref{eq1.2}).
All three shall prove to be useful throughout the paper.
Next, we prove asymptotics  of steady-state  logistic equations with respect to the
increasing power of the absorption term. As a by-product, we get an existence result for the obstacle problem (\ref{eq1.2}). As in Sections \ref{sec3} and \ref{sec4}, we
assume that (H1), (H2) are in force and $D$ is a bounded Lipschitz domain.

\subsection{Obstacle problem}

\begin{definition}
\label{def2.1}
We say that  $u\in H^{\alpha}_0(D)$  is a {\em weak solution} to (\ref{eq1.2}) if $0<u\le \mathbb I_{D\setminus \overline D_0}$ a.e. and for every $\eta\in H^{\alpha}_0(D)$ such that $\eta\le \mathbb I_{D\setminus\overline D_0}$ a.e. we have
\begin{equation}
\label{eq3.dvi}
\EE_D(u,\eta-u)\ge a(u,\eta-u).
\end{equation}
\end{definition}

\begin{proposition}
\label{prop2.1}
Assume that $u$ is a quasi-continuous bounded strictly positive function on $D$ such that $u\le \mathbb I_{D\setminus\overline D_0}$ a.e. Then the following statements are equivalent:
\begin{enumerate}
\item[\rm(i)] $u$ is a weak solution to \mbox{\rm(\ref{eq1.2})}.
\item[\rm(ii)] There exists a positive  $\mu\in S_0(D)$ such that
\begin{enumerate}
\item[\rm(a)] $\EE_D(u,\eta)=a(u,\eta)-\int_D\tilde \eta\,d\mu,\quad  \eta\in H^{\alpha}_0(D)$,
\item[\rm(b)] $\int_D(u-\eta)\,d\mu=0$ for every quasi-continuous  function $\eta$ on $D$ such that $u\le\eta\le \mathbb I_{D\setminus\overline D_0}$ a.e.,
\end{enumerate}
\item[\rm(iii)] There exists a c\`adl\`ag process $M$ with $M_0=0$  and a positive measure $\mu\in \MM_{0,b}(D)$ such that
\begin{enumerate}
\item[\rm(a)] $M$ is a uniformly integrable martingale under the measure $P_x$ for q.e. $x\in D$, and 
\[
u(X_t)=a\int_t^{\tau_D}u(X_r)\,dr-\int_t^{\tau_D}\,dA^\mu_r-\int_t^{\tau_D}\,dM_r,\quad t\le\tau_D,\,\,\, P_x\mbox{-a.s.},\quad\mbox{q.e.}\,\,\, x\in D.
\]
\item[\rm (b)] For any quasi-continuous function $\eta$ on $D$ such that $u\le\eta\le \mathbb I_{D\setminus\overline D_0}$ a.e.,
\[
\int_0^{\tau_D}(\eta-u)(X_r)\,dA^\mu_r=0\quad P_x\mbox{-a.s.}
\]
for q.e. $x\in D$.
\end{enumerate}
\end{enumerate}
\end{proposition}
\begin{proof}
The equivalence of (iii) and  (ii)  follows from \cite[Proposition 3.16]{K:SM}. 
(i) $\Rightarrow$ (ii).  For  $\xi\in C_c^\infty(D)$, we let
\[
I(\xi):=-\EE_D(u,\xi)+a(u,\xi)_{L^2(D;m)}.
\]
By (\ref{eq3.dvi}), $I(\xi)\ge 0$ for $\xi\ge 0$. Therefore, by  Riesz's theorem, there exists a positive Radon measure $\nu$ on $D$ such that $I(\xi)=\int_D\xi\,d\mu$, $\xi\in C_c^\infty(D)$. Hence
\begin{equation}
\label{eq3.dvi1}
\EE_D(u,\eta)=a(u,\eta)_{L^2}-\int_D \eta\,d\mu,\quad \eta\in C_c^\infty(D).
\end{equation}
From this one can easily conclude that $\mu\in S_0(D)$ and that  the above equation holds for any quasi-continuous $\eta\in H^{\alpha}_0(D)$.
By  (\ref{eq3.dvi}) and  (\ref{eq3.dvi1}),
\begin{equation}
\label{eq4.3}
\int_D(\eta-u)\,d\mu\le 0
\end{equation}
for any quasi-continuous  $\eta\in H^{\alpha}_0(D)$ such that $\eta\le \mathbb{I}_{D\setminus D_0}$ a.e.
Thus, (ii)(b) follows. The implication (ii) $\Rightarrow$ (i) is trivial. 
\end{proof}

\begin{proposition}
\label{prop.236dj}
If $u$ is a  weak solution to \mbox{\rm(\ref{eq1.2})}, then $u\in C_0(D)$.
\end{proposition}
\begin{proof}
Follows from \cite[Proposition 4.2]{K:arx1}.
\end{proof}

\subsection{Existence and asymptotics}

Let us recall the definition of a weak solution to
(\ref{eq1.1}) (see Remark \ref{rem4.2}).
\begin{definition}
We say that a strictly positive function $u_p\in H^{\alpha}_0(D)$ is a {\em weak
solution} to (\ref{eq1.1}) if
\begin{equation}
\label{eq5.0} \EE_D(u_p,\eta)= (au_p,\eta)-(bu_p^p,\eta),\quad
\eta\in H^{\alpha}_0(D).
\end{equation}
\end{definition}

\begin{theorem}
\label{th5.1}
\begin{enumerate}
\item[\rm(i)] For every $a\in (\lambda_1^D,\lambda_1^{D_0})$ there exists
a unique bounded weak solution $u$ to {\rm (\ref{eq1.2})}.

\item[\rm(ii)]Let $u_p$, $p>1$, be a weak solution to  {\rm{(\ref{eq1.1})}}. Then
\[
 \|u_p-u\|_\infty+\EE_D(u_p-u,u_p-u)\rightarrow 0 \quad\mbox{\rm as }p\rightarrow \infty.
\]
\end{enumerate}
\end{theorem}
\begin{proof}
Choose $\{q_k\}\subset\BB^+_b(D)$ so that supp$[q_k]\subset
D\setminus \overline D_0,\, k\ge 1$, and for every compact $K\subset
D\setminus \overline D_0$, $\inf_{x\in K}q_k(x)\nearrow \infty$.
Let  $\psi_k$ be the principal
eigenfunction for $-(\Delta^{\alpha})_{|D}+q_k$. Then for any $c\ge 0,\, k\ge 1$,
\begin{equation}
\label{eq5.1.1.2}
-(\Delta^{\alpha})_{|D}(c\psi_k)=a(c\psi_k)-b(c\psi_k)^p+(\lambda_1^D[q_k]-a)(c\psi_k)+b(c\psi_k)^p-q_k(c\psi_k).
\end{equation}
By the fact that $a<\lambda_1^{D_0}$ and Theorem \ref{th3.1}, there
exists $k_0\ge 1$ such that $\lambda_1^D[q_{k_0}]-a\ge 0$. Observe that
\[
b(c\psi_{k_0})^p-cq_{k_0}\psi_{k_0}\ge \mathbf{1}_{K_{k_0}}c\psi_{k_0}(c^{p-1}\inf_{x\in K_{k_0}}b(x)\inf_{x\in K_{k_0}}\psi_{k_0}(x)-\sup_{x\in D}q_{k_0}(x)),
\]
where $K_{k_0}=\mbox{supp}[q_{k_0}]$. Since $K_{k_0}$ is compact
and $K_{k_0}\subset D\setminus \overline D_0$,  (H1) implies that
\[
d_{k_0}:=\inf_{x\in K_{k_0}}b(x)\inf_{x\in K_{k_0}}\psi_{k_0}(x)>0.
\]
Hence, since $q_{k_0}$ is bounded, there exists $c_{k_0}$
such that for any $p\ge 2$,
\[
b(c_{k_0}\psi_{k_0})^p-c_{k_0}q_{k_0}\psi_{k_0}\ge 0.
\]
Therefore $c_{k_0}\psi_{k_0}$ is a supersolution to (\ref{eq1.1}).
Since $\lambda_1^D<a$, we easily conclude that $\varepsilon
\varphi^D_1$ is a subsolution to (\ref{eq1.1}) for a sufficiently
small $\varepsilon >0$. Moreover, by \cite[Theorem 3.4]{DS}, for a
sufficiently small $\varepsilon>0$, $\varepsilon \varphi^D_1\le
c_{k_0}\psi_{k_0}$. By Propositions \ref{prop3.3} and
\ref{prop3.2},
\begin{equation}
\label{eq5.1}
\varepsilon\varphi^D_1(x)\le u_p(x)\le c_{k_0}\psi_{k_0}(x),\quad x\in D,\quad p\ge 2.
\end{equation}
By the definition of a weak solution to (\ref{eq1.1}),
\begin{equation}
\label{eq5.2} \EE_D(u_p,\eta)+\int_D\eta\,d\mu_p=a(u_p,\eta),\quad
\eta\in D(\EE_D),
\end{equation}
where $\mu_p=bu^p_p\cdot m$. Taking $\eta=u_p$ as a test function
and using (\ref{eq5.1}), we conclude that
\begin{equation}
\label{eq5.4}
\sup_{p\ge 2}\EE_D(u_p,u_p)<\infty,\qquad
\sup_{p\ge 2}\int_D bu^{p+1}_p\,dm<\infty.
\end{equation}
From this and (\ref{eq5.2}) we deduce that $\sup_{p\ge
2}\|\mu_p\|_{H^{-\alpha}(D)}<\infty$ (cf. \eqref{eq.ppt2}). Therefore, there exists $u\in
H^{\alpha}_0(D)$ and $\mu\in  S_0(D)$ such that, up to a subsequence,
$\mu_p\rightarrow \mu$ weakly in $H^{-\alpha}(D)$ and $u_p\rightarrow
u$ weakly in $H^{\alpha}_0(D)$. Moreover, since $H^{\alpha}_0(D)$ is compactly
embedded in $L^q(D;m)$ for  $q\in [1,2d(d-2\alpha))$ (cf. Section \ref{sec2.1}), up to a subsequence,
$u_p\rightarrow u$  a.e. From the second inequality in
(\ref{eq5.4}) and (H1), we easily deduce  that
\begin{equation}
\label{eq5.5} u\le \mathbb{I}_{D\setminus\overline D_0}\quad
m\mbox{-a.e.}
\end{equation}
 Taking $\eta=u_p-u$ in (\ref{eq5.2}) we get
\begin{equation}
\label{eq5.3}
\EE_D(u_p-u,u_p-u)+\EE_D(u,u_p-u)+\int_D(u_p-u)\,d\mu_p=a(u_p,u_p-u).
\end{equation}
Next, by (\ref{eq5.5}),
\[
\int_D(u_p-u)\,d\mu_p\ge \int_{\{u_p\le u\}}bu^p_p(u_p-u) \ge
-\|b\|_\infty\int_{\{u_p\le u\}}|u_p-u|. \quad
\]
Substituting  into \eqref{eq5.3}, and  using  weak convergence of $\{u_p\}$ in
$D(\EE_D)$ and
strong convergence of $\{u_p\}$ in $L^q(D;m)$  we conclude that, up
to a subsequence,
\[
\EE_D(u_p-u,u_p-u)\rightarrow 0.
\]
Let $\eta\in H^{\alpha}_0(D)$ be such that  $\eta\le
\mathbb{I}_{D\setminus\overline D_0}$  a.e., and let $\gamma\in
(0,1)$. By (H1),
\begin{align}
\label{eq5.tm} \int_D(\gamma\eta-u_p)\,d\mu_p&\le
\int_{\{\gamma\eta\ge u_p\}}(\gamma\eta-u_p)bu^p_p\,dm\nonumber\\
&\le \|b\|_\infty\int_{\{\gamma\eta\ge
u_p\}}|\gamma\eta-u_p|\gamma^p\,dm\rightarrow 0.
\end{align}
Therefore, by already  proved convergence properties  of $\{u_p\}$, and
(\ref{eq5.2}), we get that for every $\eta\in H^{\alpha}_0(D)$  such that
$\eta\le \mathbb{I}_{D\setminus\overline D_0}$  a.e.,
\begin{equation}
\label{eq5.6}
\EE_D(u,\eta-u)\ge a(u,\eta-u).
\end{equation}
This together with (\ref{eq5.5}) implies that   $u$ is a weak solution to
(\ref{eq1.2}). Moreover, by Proposition \ref{prop.236dj}, $u$ is continuous. By the uniqueness result for  (\ref{eq1.2}) (see
\cite{K:arx1}), $\EE_D(u_p-u,u_p-u)\rightarrow 0$.

As for the uniform convergence in (ii), by Proposition
\ref{prop2.1}(iii), and  Lemma \ref{lm2.4.1},
\begin{align*}
u_p(X_t)-u(X_t)&= a\int_t^{\tau_D}(u_p-u)(X_r)\,dr
-\int_t^{\tau_D}bu_p^p(X_r)\,dr\\&\quad +\int_t^{\tau_D}\,dA^\mu_r-\int_t^{\tau_D}\,d(M^p_r-M_r),\quad t\in [0,\tau_D],\, P_x\mbox{-a.s.}
\end{align*}
for some martingales $M^p, M$, and q.e. $x\in D$.
By It\^o's formula, (H1) and (\ref{eq.estgfu2}),
\begin{align}
\label{eq.ptg1}
\nonumber|u_p(x)-u(x)|^2&= 2a\mathbb E_x\int_0^{\tau_D}|u_p-u|^2(X_r)\,dr
-2\mathbb E_x\int_0^{\tau_D}(u_p-u)(X_r)bu_p^p(X_r)\,dr\\&\quad\nonumber
+2\mathbb E_x\int_0^{\tau_D}(u_p-u)(X_r)\,dA^\mu_r\\&\nonumber \le
2a\mathbb E_x\int_0^{\tau_D}|u_p-u|^2(X_r)\,dr+2 \|b\|_\infty
\mathbb E_x\int_0^{\tau_D}|u_p-u|(X_r)\,dr\\&\nonumber\quad
+2\mathbb E_x\int_0^{\tau_D}|u_p-u|(X_r)\,dA^\mu_r \\& \le
c(2a+2\|b\|_\infty)\|u_p-u\|_{L^q(D;m)}
+2\mathbb E_x\int_0^{\tau_D}|u_p-u|(X_r)\,dA^\mu_r
\end{align}
for any  $q\in (1,d/(d-2\alpha))$, and with $c$ depending only on $q,D,\alpha$ and $d$.
Set $\theta:= c_{k_0}\|\psi_{k_0}\|_\infty$. By  \eqref{eq5.1} and ultracontractivity of $(P^D_t)_{t\ge 0}$ (cf. \eqref{eq.iuc12u}), for $h>0$ we have
\begin{align}
\label{eq.ptg2}
\nonumber&\mathbb E_x\int_0^{\tau_D}|u_p-u|(X_r)\,dA^\mu_r\\
&\nonumber\quad=\mathbb E_x\int_h^{\infty}\mathbf{1}_{\{r<\tau_D\}}|u_p-u|(X_r)\,dA^\mu_r
+ \mathbb E_x\int_0^h\mathbf{1}_{\{r<\tau_D\}}|u_p-u|(X_r)\,dA^\mu_r\\
&\nonumber\quad=
\int_D\int_h^{\infty}|u_p-u|(y)p_D(t,x,y)\,dt\,\mu(dy)+\mathbb E_x\int_0^h
\mathbf{1}_{\{r<\tau_D\}} |u_p-u|(X_r)\,dA^\mu_r\\
& \quad\le \beta_h
\frac{e^{-h\lambda^D_1 }\|\varphi_1^D\|^2_\infty}{\lambda_1^D}
\int_D|u_p-u|(y)\,\mu(dy)+2\theta\mathbb E_x\int_0^h
\mathbf{1}_{\{r<\tau_D\}}\,dA^\mu_r.
\end{align}
Since  $\mu\in S_0(D)$ as  shown above, we have   
\begin{equation}
\label{eq.ptg3}
\int_D|u_p-u|(y)\,\mu(dy)\le c\EE_D(u_p-u,u_p-u).
\end{equation}
By Proposition \ref{prop2.1}(iii), for every $x\in D$,
\begin{align}
\label{eq.ptg4}
\nonumber \mathbb E_x\int_0^h \mathbf{1}_{\{r<\tau_D\}}\,dA^\mu_r &=a\mathbb E_x\int_0^h
\mathbf{1}_{\{r<\tau_D\}}u(X_r)\,dr+\mathbb E_x[u(X_h)\mathbf{1}_{\{h<\tau_D\}}]-u(x)\\&
\le a\theta h+\|P^D_h(u)-u\|_\infty.
\end{align}
Since $(P^D_t)$ is Fellerian, $\|P^D_h(u)-u\|_\infty\rightarrow
0$ as $h\searrow 0$. Consequently, putting together \eqref{eq.ptg1}--\eqref{eq.ptg4}, and  the already proven convergence
properties of $(u_p)$, 
we conclude that
$\|u_p-u\|_\infty\rightarrow 0$ as $p\rightarrow \infty$.
\end{proof}

\section{Parabolic equations: existence and probabilistic interpretation}

Let $m_1$ be the Lebesgue measure on $\BR^{d+1}$. Set $D_T=(0,T)\times D$.
Let $\langle\cdot,\cdot\rangle$ denote the  duality pairing between
$H^{\alpha}_0(D)$ and its dual space $H^{-\alpha}(D)$.  Set
\[
\WW=\{u\in L^2(\BR;
H^{\alpha}_0(D):\frac{d u}{d t}\in L^2(\BR;
H^{-\alpha}(D))\},
\]
\[
\WW(0,T)=\{u\in L^2(0,T; H^{\alpha}_0(D)):\frac{d u}{d t}\in L^2(0,T; H^{-\alpha}(D))\},
\]
and define a bilinear form $\BB^D$ by
\[
\mathcal{B}^D(u,v)=\left\{
\begin{array}{l}\int_\BR\langle-\frac{d u}{d t},
v\rangle\,dt +\int_\BR\EE_D(u,v)\, dt,
\quad u\in\WW,v\in L^2(\BR; H^{\alpha}_0(D)),\smallskip \\
\int_\BR\langle-\frac{d v}{d t},u\rangle\,dt
+\int_\BR\EE_D(u,v)\, dt,\quad u\in L^2(\BR; H^{\alpha}_0(D)),v\in\WW,
\end{array}
\right.
\]

Let $\mathfrak X^D= ((\mathscr X^D_t)_{t\ge0}, (P_{s,x})_{(s,x)\in \BR\times D},(\mathscr F_t)_{t\ge0})$ be a Hunt process associated with the form $\mathcal B^D$ (see \cite[Theorem 6.3.1]{Oshima}). In fact (see \cite[Theorem 6.3.1]{Oshima} again)
\[
\mathscr X^D_t= (\upsilon(t),X^D_{\upsilon(t)}),\quad t\ge 0,
\]
where $\upsilon(t)$ is the uniform motion to the right, i.e. $\upsilon(t)=\upsilon(0)+t$ and $\upsilon(0)=s$ $P_{s,x}$-a.s. Moreover,  $X^D$ is a c\`adl\`ag process such that  for any Borel subset $B$ of $D$,
\[
P_{s,x}(X^D_t\in B)=\int_B p_D(t-s,x,y)\,dy,\quad x\in D,\,s<t.
\]
It follows that  for fixed $s\ge 0$ process $t\mapsto X^D_{s+t}$ under measure  $P_{s,x}$, for $x\in D$,
agrees with  the process $\mathbb X^D$ introduced in Section \ref{sec2.2}.

As in \cite[Section 6.2]{Oshima}, we define a Choquet capacity naturally associated with the form $\BB^D$. We shal denote it by $\mbox{Cap}_1$. Then, as in the case of  the form
$\EE_D$, we define quasi-notions associated with $\mbox{\rm Cap}_1$ ($\mbox{Cap}_1$-q.e., $\BB^D$-quasi-continuity, $\BB^D$-smooth measures).
We denote by  $\MM_{0,b}(\BR\times D)$  the set of $\BB^D$-smooth bounded measures on $\BR\times D$, and for fixed $T>0$,
we denote by $\MM_0(D_T)$ the subset of $\MM_{0,b}(\BR\times D)$ consisting of measures $\mu$ such that $\mu((\BR\times D)\setminus D_T)=0$. By \cite[Proposition 4.1]{KR:JEE}, for every positive smooth measure $\mu$ on $D_T$ there exists a unique PNAF (positive natural additive functional)  $A^\mu$ of $\mathfrak X^D$ in the Revuz duality with $\mu$.

In the sequel, for a function $v$ on $D_T$, we let
\[
v^{(T)}(t,x):= v(T-t,x),\quad (t,x)\in D_T,
\]
and for a given measure $\mu\in\MM_{0,b}(D_T)$, we denote by $\mu^{(T)}$
the measure on $D_T$ given by
\[
\int_{D_T}\eta\,d\mu^{(T)}=\int_{D_T}\eta^{(T)}\,d\mu,\quad \eta\in C_b(D_T).
\]
Recall that, starting from Section \ref{sec3}, we assume in the paper that
conditions  (H1),(H2) (see Introduction) are satisfied.
In the sequel,  we frequently use, without special mention,  that $\WW(0,T)\subset C([0,T];L^2(D;m))$   (see e.g. \cite[Remarque 1.2, page 156]{Lions}).

\subsection{Probabilistic interpretation of solutions to linear equations}

Let $\varphi\in L^2(D;m)$ and $f\in L^2(D_T;m_1)$.
Consider  the following linear equation.
\begin{equation}
\label{eq1.3test2} \left\{
\begin{array}{l}\frac{d v}{d t}-\Delta^{\alpha} v= f,\quad\mbox{in}\,\,\, D\times (0,\infty),\medskip\\
\,v=0,\quad\mbox{in}\,\,\, (\BR^d\setminus D)\times (0,\infty),
\medskip \\
\,v(0,\cdot)=\varphi,\quad\mbox{in}\,\,\, D.
\end{array}
\right.
\end{equation}

\begin{definition}
We say that a bounded function  $v\in\WW(0,T)$ is a {\em weak solution} to (\ref{eq1.3test2}) on $[0,T]$ if $v(0,\cdot)=\varphi$ and
for every $\eta\in L^2(0,T;H^{\alpha}_0(D))$,
\begin{equation}
\label{eq6.0}
\int_0^s\Big\langle\frac{d v}{dt},\eta\Big\rangle\,dt+\int_0^s\EE_D(v,\eta)\,dt= \int_0^s(f,\eta)\,dt,\quad s\in (0,T).
\end{equation}
\end{definition}

\begin{proposition}
\label{prop.ppt1}
Let $\varphi\in L^2(D;m)$ and $f\in L^2(D_T;m_1)$.
\begin{enumerate}[\rm(i)]
\item  $v$ is a  weak solution to \eqref{eq1.3test2} if and only  if
\begin{equation}
\label{eq6.1ppt}
v^{(T)}(s,x)=\mathbb E_{s,x}\varphi(X^D_T)+\mathbb E_{s,x}\int_0^{T-s} f^{(T)}(\mathscr X^D_r)\,dr,\quad\mbox{a.e.}\,\,\, (s,x)\in D_T.
\end{equation}

\item There exists a unique weak solution $v$ to \eqref{eq1.3test2} on $[0,T]$. 
Moreover,  $v$ is a strong solution to \eqref{eq1.3test2}, i.e. $v$ is  absolutely continuous on $[0,T]$,  $\frac{dv}{dt}\in L^1(0,T;L^2(D;m))$, $v(t)\in D((\Delta^{\alpha})_{|D})$ a.e. $t\in [0,T]$, and 
\begin{equation}
\label{eq6.1pptpizz}
\frac{dv}{dt}(t)-(\Delta^{\alpha})_{|D}v(t)=f(t),\quad\mbox{a.e.}\,\,\, t\in [0,T].
\end{equation}

\item Let $\tilde v^{(T)}(s,x)$ be equal to the right-hand side of \mbox{\rm(\ref{eq6.1ppt})}
if it is finite, and $\tilde v^{(T)}(s,x)=0$ otherwise. Then there exists a c\`adl\`ag process $M$ with $M_0=0$ such that $M$ is an $(\mathscr F_t)_{t\ge0}$-martingale under the measure $P_{s,x}$ and
\[
\tilde v^{(T)}(\mathscr X^D_t)=\varphi(X^D_T)+\int_t^{T-s} f^{(T)}(\mathscr X^D_r)\,dr -\int_t^{T-s}\,dM_r,\quad t\in [0,T-s],\quad P_{s,x}\mbox{-}a.s.,
\]
for every $(s,x)\in D_T$ such that $\mathbb E_{s,x}|\varphi(X^D_T)|+\mathbb E_{s,x}\int_0^{T-s} |f^{(T)}(\mathscr X^D_r)|\,dr<\infty$.
\end{enumerate}
\end{proposition}
\begin{proof}
(i) and (iii) follow from \cite[Theorem 3.7, Theorem 5.8]{K:JFA}. Observe that \eqref{eq6.1ppt} means that
$v$ is a {\em mild solution} to \eqref{eq1.3test2}. Therefore, by \cite[Theorem 8.2.1]{Vrabie}, $v$ is a {\em strong solution}
to \eqref{eq1.3test2}.
\end{proof}

\begin{remark}
For brevity (and in light of Proposition \ref{prop.ppt1}(ii)), we frequently write that $v$ is a  weak (strong) solution to
\[
\frac{dv}{dt}-(\Delta^{\alpha})_{|D}v=f,\quad v(0)=\varphi
\]
instead of writing that it is a weak (strong) solution to \eqref{eq1.3test2}.
\end{remark}

\begin{remark}
\label{rem.bsde}
The displayed formula in Proposition \ref{prop.ppt1}(iii) says that the pair of processes $(\tilde v^{(T)}(\mathscr X^D), M)$
is a solution of the so called Backward Stochastic Differential Equation (BSDE) with terminal condition $\varphi(X^D_T)$,
and right-hand side $ f^{(T)}(\mathscr X^D)$ (see \cite{KR:JFA}).

\end{remark}

\subsection{Existence for parabolic logistic equations}

\begin{theorem}
\label{th6.impar5}
For every $p>0$ there exists a unique bounded weak solution $v_p$ to \mbox{\rm(\ref{eq1.3})}.
Moreover,   
$\frac{d v_p}{d t}, (\Delta^{\alpha})_{|D}v_p\in C((0,T];L^2(D;m))$.
\end{theorem}
\begin{proof}
The existence of a bounded weak solution to \eqref{eq1.3} follows from  \cite[Theorem 3.7,Theorem 5.4,Theorem 5.8]{K:JFA}.
The uniqueness part is a standard result (see e.g. \cite[Proposition 3.6]{K:JFA}). 
By \cite{LP}, the semigroup $(P^D_t)$ is analytic on $L^2(D;m)$.
Therefore, by \cite[Theorem 3.1, Section 4.3]{Pazy}, $[\varepsilon,T]\ni t\mapsto v_p(t)\in L^2(D;m)$ is $\frac12-$H\"older
continuous for any $\varepsilon>0$. 
Now, the asserted  regularity follows from \cite[Theorem 3.5, Section 4.3]{Pazy}.
\end{proof}

\begin{remark}
\label{rem.th6.impar5}
Let  $v_p\in\WW(0,T)$ be a bounded function.
In light of the above theorem,   (\ref{eq6.0}) is equivalent to each of the following statements:  (a) for any  $t\in (0,T)$ and  $\eta\in H^{\alpha}_0(D)$,
\begin{equation}
\label{eq6.0a}
\Big(\frac{dv_p}{dt}(t),\eta\Big)+\EE_D(v_p(t),\eta)= a(v_p(t),\eta)-(bv_p^p(t),\eta).
\end{equation}
(b) $\frac{d v_p}{dt}\in L^2(D_T)$,   $v_p(t)\in D((\Delta^{\alpha})_{|D})$ a.e. $t\in (0,T)$, $v_p(0)=\varphi$, and 
\[
\frac{d v_p}{dt}(t)-(\Delta^{\alpha})_{|D}v_p(t)=av_p(t)-bv_p^p(t),\quad\mbox{a.e.}\,\,\, t\in (0,T).
\]
\end{remark}

\section{Obstacle problem and asymptotics as $p\rightarrow \infty$ for parabolic equations}

Let $\WW_T(0,T)=\{u\in \WW(0,T):u(T)=0\}$. In this section, we shall
prove asymptotics, with respect to the increasing power of the absorption term,  for parabolic logistic equations. To this end, as in the elliptic case, we  begin with providing some  equivalent formulations of the parabolic obstacle problem (\ref{eq1.4}). We shall also show some regularity results for weak solutions to (\ref{eq1.4}).

Recall that we assumed that (H1), (H2) are satisfied and $D$ is a bounded Lipschitz domain.

\subsection{Obstacle problem}

\begin{definition}
We say that  $v\in C([0,T]; L^2(D;m))\cap L^2(0,T; H^{\alpha}_0(D))$ is a {\em weak solution} to (\ref{eq1.4}) on $[0,T]$ if
\begin{enumerate}
\item[{\rm{(i)}}] $v\le \mathbb{I}_{D\setminus \overline D_0}$ a.e. and $v(0,\cdot)=\varphi$  a.e.,
\item[{\rm{(ii)}}] For every $\eta\in \WW(0,T)$ such that $\eta\le \mathbb{I}_{D\setminus \overline D_0 }$ a.e. we have
\begin{align*}
\int_0^s\Big\langle \frac{d\eta}{dt},\eta & -v\Big\rangle\,dt+\int_0^s\EE_D(v,\eta-v)\,dt\ge \int_0^s(av,\eta-v)\,dt\\&\quad
+\frac12\|\eta(s)-v(s)\|^2_{L^2(D;m)}-
\frac12\|\eta(0)-\varphi\|^2_{L^2(D;m)},\quad s\in [0,T].
\end{align*}
\end{enumerate}
\end{definition}

\begin{remark}
\label{rem.eqovsop1}
It is an elementary check that if additionally to regularity of $v$ required in the definition of weak solution to \eqref{eq1.4},
we know that $v\in\WW(0,T)$, then (ii) is equivalent to the following condition:  for every $\eta\in L^2(0,T;H^{\alpha}_0(D))$ such that $\eta\le \mathbb{I}_{D\setminus \overline D_0 }$ a.e. we have 
\begin{align*}
\int_0^T\Big\langle \frac{dv}{dt},\eta  -v\Big\rangle\,dt+\int_0^T\EE_D(v,\eta-v)\,dt\ge \int_0^T(av,\eta-v)\,dt,\quad v(T)=\varphi.
\end{align*}
Furthermore, if we know that $\frac{dv}{dt}\in L^2(D_T;m_1)$, then the above condition is equivalent to the following one:
for every $\eta\in H^{\alpha}_0(D)$ such that $\eta\le \mathbb{I}_{D\setminus \overline D_0 }$ $m$-a.e. we have 
\begin{align*}
\Big(\frac{dv}{dt}(t),\eta  -v(t)\Big)+\EE_D(v(t),\eta-v(t))\ge (av(t),\eta-v(t))\,dt,\,\,\, \mbox{a.e.}\,\, t\in [0,T],\,\,v(T)=\varphi.
\end{align*}

\end{remark}
Before we proceed to the next result, we recall some auxiliary  notions. 
A function $u:D_T\to \BR$ is called {\em quasi-c\`adl\`ag }(see \cite{K:JEE})
if for q.e. $(s,x)\in D_T$ process $u(\mathscr X^D)$ is c\`adl\`ag on $[0,T-s]$ under measure $P_{s,x}$.
In \cite{K:JEE} (see definition on page 704 in \cite{K:JEE} and comments following it), we introduced a notion of a {\em probabilistic solution} to \eqref{eq1.4} according to which, $u:D_T\to\BR$
is a solution to \eqref{eq1.4} if $u\le \mathbb I_{D\setminus\overline D_0}$ a.e., there exists a positive smooth measure $\nu$
on $D_T$ such that
\begin{equation}
\label{eq.pp1}
u(s,x)=\mathbb E_{s,x}\varphi(X^D_T)+a\mathbb E_{s,x}\int_0^{T-s}u^{(T)}(\mathscr X^D_r)\,dr-\mathbb E_{s,x}\int_0^{T-s}\,dA^{\nu^{(T)}}_r,
\end{equation} 
and for any quasi-c\`adl\`ag function $\eta$ such that $u\le \eta\le \mathbb I_{D\setminus\overline D_0}$ a.e., we have
\begin{equation}
\label{eq.pp2}
\int_{D_T}(\hat \eta-\hat u)\,d\nu=0.
\end{equation} 
Here $\hat u,\hat \eta$ are {\em precise versions} of $u,\eta$, respectively. 
By the very definition of the precise version (see  definition on page 692 in \cite{K:JEE}; see also comments preceding Lemma 5.1 in \cite{Oshima2}), if $u,\eta$ are quasi-continuous, then
$\hat u=u$ and $\hat \eta=\eta$. In this case \eqref{eq.pp2} may be  replaced by the following condition:
for any quasi-continuous function $\eta$ such that $u\le \eta\le \mathbb I_{D\setminus\overline D_0}$ a.e., we have
\begin{equation}
\label{eq.pp3}
\int_{D_T}(\eta-u)\,d\nu=0.
\end{equation} 
By \cite[Theorem 5.3]{K:JEE} (see also the comments after definition on page 704 in \cite{K:JEE}),
formulation (7.1),(7.3) guarantees  the uniqueness of a quasi-continuous probabilistic solution to \eqref{eq1.4}.

\begin{proposition}
\label{prop2.1v}
(1) Assume that $v$ is  quasi-continuous and $v\in L^2(0,T;H^{\alpha}_0(D))$. Then the following statements (i)--(iii) are equivalent.
\begin{enumerate}
\item[\rm(i)] $v$ is a weak solution to \mbox{\rm(\ref{eq1.4})} on $[0,T]$.

\item[\rm (ii)] There exists a positive  $\nu\in \MM_{0,b}(D_T)$ such that
\begin{enumerate}
\item[\rm(a)] $\int_0^T\big\langle \frac{d\eta}{d t},v\big\rangle\,dt+\int_0^T \EE_D(v,\eta)\,dt =(\varphi,\eta(0))+a\int_0^T (v,\eta)\,dt -\int_{D_T}\tilde \eta\,d\nu$  for every bounded $\eta\in \WW_T(0,T)$,
\item[\rm (b)] $\int_{D_T}(v-\eta)\,d\nu=0$ for every quasi-continuous  function $\eta$ on $D_T$ such that $v\le\eta\le \mathbb I_{D\setminus\overline D_0}$ a.e.
\end{enumerate}

\item[\rm (iii)] There exists a c\`adl\`ag process $M$ with $M_0=0$  and a positive measure $\nu\in \MM_{0,b}(D_T)$ such that
\begin{enumerate}
\item[\rm(a)] $M$ is a uniformly integrable $(\mathscr F_t)_{t\ge 0}$-martingale under the measure $P_{s,x}$ for q.e. $(s,x)\in D_T$, and
\begin{align*}
v^{(T)}(\mathscr X^D_t)&=\varphi(X^D_T)+a\int_t^{T-s}v^{(T)}(\mathscr X^D_r)\,dr-\int_t^{T-s}\,dA^{\nu^{(T)}}_r\\&\quad
-\int_t^{T-s}\,dM_r,\quad t\le T-s,\quad  P_{s,x}\mbox{-}a.s.,\quad\mbox{q.e.}\,\,\, (s,x)\in D_T.
\end{align*}

\item[\rm(b)] For every quasi-continuous function $\eta$ on $D_T$ such that $v^{(T)}\le\eta\le \mathbb I_{D\setminus\overline D_0}$ a.e.,
\[
\int_0^{T-s}(\eta-v^{(T)})(\mathscr X_r)\,dA^{\nu^{(T)}}_r=0,\quad P_{s,x}\mbox{-a.s.}\quad\mbox{q.e.}\,\,\, (s,x)\in D_T.
\]
\end{enumerate}
\end{enumerate}

(2) Let $v_n\in\WW(0,T)$ be a quasi-continuous version of a weak solution to 
\begin{equation}
\label{eq7.wcom1}
\frac{dv_n}{dt}-(\Delta^{\alpha})_{|D} v_n=av_n-n(v_n-\mathbb{I}_{D\setminus\overline D_0})^+,\quad v_n(0)=\varphi.
\end{equation}
Then, defining  $\bar v:= \liminf_{n\rightarrow \infty} v_n$, we get that $\bar v^{(T)}$  satisfies (iii)(a).
Moreover, if $\bar v$ is quasi-continuous, then $\bar v^{(T)}$ satisfies (iii)(b).
\end{proposition}
\begin{proof}
By  the proof of \cite[Theorem 5.4]{K:JEE} (see Eq. (5.4) in \cite{K:JEE}), $\bar v^{(T)}(\mathscr X^D)$ satisfies (iii)(a),
and \eqref{eq.pp2},  where $\bar v$ is defined in (2).  If $\bar v$ is quasi-continuous, then by the comment following \eqref{eq.pp2},
we have that \eqref{eq.pp3} holds. Therefore, by the definition of the Revuz duality, we get (iii)(b).
This completes the proof of (2).

From  \cite[Proposition 3.6, Theorem 3.7, Theorem 5.8]{K:JFA} it follows that (ii)(a) and (iii)(a) are equivalent, whereas from the definition of the Revuz duality between $\nu$ and $A^\nu$
it follows that (ii)(b) and (iii)(b) are equivalent. Therefore (ii) is equivalent to (iii).
The proof of (1) shall be  completed by showing that (i) is equivalent to (iii). To do this end, we first note that by  \cite[Theorem 6.2, Chapter 3]{Lions}, there exists a unique weak solution $v$ to (\ref{eq1.4}), and
it is  the  limit of functions $v_n$ solving (\ref{eq7.wcom1}).
Suppose that  $v$ is a solution to  $(iii)$. 
By  \cite[Theorem 5.4]{K:JEE}, $\bar v^{(T)}(\mathscr X^D_t)=v^{(T)}(\mathscr X_t),\, t\in [0,T-s],\, P_{s,x}$-a.s. for q.e.  $(s,x)\in D_T$,
where $\bar v$ is defined in (2). Thus, $v$ ($=\bar v$ a.e.) is a weak solution to \eqref{eq1.4}.
Suppose now, that $v$ is a weak solution to \eqref{eq1.4}.  As already mentioned $v=\bar v$ a.e. 
Therefore, for any $(s,x)\in D_T$,
\[
E_{s,x}\int_0^{T-s}|\bar v-v|(X_r)\,dr=\int_s^T\int_Dp_D(t-s,x,y)|\bar v-v|(y)\,dy=0.
\]
By (2), $\bar v^{(T)}$ satisfies (iii)(a), and so  $\bar v^{(T)}(\mathscr X^D)$  is a c\`adl\`ag process.
Since $v$ is assumed to be quasi-continuous, we have that $v^{(T)}(\mathscr X^D)$ is a c\`adl\`ag process too (see the comments preceding Lemma 5.1 in \cite{Oshima2}). Therefore, using the above equation, we conclude that  $\hat v^{(T)}(\mathscr X^D_t)=v^{(T)}(\mathscr X^D_t),\, t\in [0,T-s],\, P_{s,x}$-a.s. for q.e. $(s,x)\in D_T$.  This implies, in particular, that $\bar v=v$ q.e., and so $\bar v$ is quasi-continuous.
As a result, applying (2), we get that $v^{(T)}(\mathscr X^D)$ satisfies (iii)(b). By \cite[Proposition 5.7]{K:JEE}, $\nu$
is bounded. So, $v$ satisfies (iii).
\end{proof}

In the sequel we will freely use, without special mention, the equivalent notions of solutions to \eqref{eq1.4}
stated in Proposition \ref{prop2.1v}, depending on the source we will refer to.

\begin{remark}
\label{rem.rbsde}
Proposition \ref{prop2.1v}(iii) says that the triple  $(\tilde v^{(T)}(\mathscr X^D), M, A^{\nu^{(T)}})$
is a solution of the so called Reflected Backward Stochastic Differential Equation (RBSDE) with terminal condition $\xi:=\varphi(X^D_T)$,
 right-hand side $f(y):=ay$, and barrier $L_t:= \mathbb I_{D\setminus\overline D_0}(\mathscr X^D_t)$ (see \cite{K:SPA}).
\end{remark}

In the following proposition, we use the notion of perfect PCAFs of $\mathfrak X^D$ (see \cite[Section IV]{BG} for the definition).

\begin{proposition}
\label{prop2.1vv}
Let $v$ be a weak solution to \mbox{\rm(\ref{eq1.4})} on $[0,T]$. Then  $v$ is quasi-continuous and there exists a perfect PCAF $\tilde A^{\nu^{(T)}}$ of $\mathfrak X^D$  such that condition (iii) of Proposition \ref{prop2.1v} holds for every $(s,x)\in D_T$.  Moreover, if $\varphi\in C_0(D)$, then $v\in C([0,T];C_0(D))$.
\end{proposition}
\begin{proof}
By \cite[Proposition 5.5]{K:JEE}), $v\le w$ , where $w$ is a weak solution of the  Cauchy-Dirichlet problem
\[
\frac{d w}{d t}-(\Delta^{\alpha})_{|D} w=a w,\quad w(0)=\varphi.
\]
It is clear that $w$ is bounded, so $v$ is bounded, too.  Set
\[
h(s,x)=1+\mathbb E_{s,x}\int_0^{T-s}v^{(T)}(\mathscr X^{D_0}_r)\,dr.
\]
By \cite[Theorem 6.3.1]{Oshima}, $h$ is quasi-continuous. By Proposition \ref{prop2.1v}(2),
\begin{align*}
v^{(T)}(s,x)&=\mathbb E_{s,x}v^{(T)}(\mathscr X^D_{\tau_{D_0}})+a\mathbb E_{s,x}\int_0^{\tau_{D_0}\wedge (T-s)}v^{(T)}(\mathscr X^D_r)\,dr-\mathbb E_{s,x}\int_0^{\tau_{D_0}\wedge (T-s)}\,dA^{\nu^{(T)}}_r\\& \le 1+a\mathbb E_{s,x}\int_0^{\tau_{D_0}\wedge (T-s)}v^{(T)}(\mathscr X^D_r)\,dr=h(s,x).
\end{align*}
Thus, $v^{(T)}\le h\le \mathbb I_{D\setminus\overline D_0}$ a.e. Therefore, in fact, $v$ is a weak solution to (\ref{eq1.4}) with $\mathbb I_{D\setminus\overline D_0}$ replaced by $h$. Now, applying  Proposition \ref{prop2.1v}(iii) (see also Remark \ref{rem.rbsde})
to $v$  with $\mathbb{I}_{D\setminus\overline D_0}$ replaced by $h$, and then  \cite[Corollary 4.4]{K:SPA}, we conclude that $A^{\nu^{(T)}}$ is continuous $P_{s,x}$-a.s.
for q.e. $(s,x)\in D_T$.  Therefore, by \cite[Theorem IV.3.8]{Stannat}, $v$ is quasi-continuous.  Set
\[
\tilde v^{(T)}(s,x)= \tilde v^{(T)}_1(s,x)-\tilde v_2^{(T)}(s,x),\quad (s,x)\in D_T,
\]
where
\[
\tilde v^{(T)}_1(s,x)=\int_D \varphi(y)p_D(s,x,y)\,dt+\int_s^{T} \int_D p_D(r-s,x,y) v^{(T)}(r,y)\,dr\,
\]
\[
\tilde v^{(T)}_2(s,x)=\int_s^{T} \int_D p_D(r-s,x,y)\,\nu^{(T)}(dr\,dy).
\]
By Proposition \ref{prop2.1v}(iii) and Revuz duality, for q.e. $(s,x)\in D_T$ we have
\[
\tilde v^{(T)}_2(s,x)=\mathbb E_{s,x}\int_0^{T-s}\,dA^{\nu^{(T)}}_r\le Ta\|v\|_\infty+\|\varphi\|_\infty.
\]
Since $\tilde v^{(T)}_2$ is   lower semi-continuous (as $p_D$ is lower semi-continuous), the above inequality holds for every $(s,x)\in D_T$. Therefore, by \cite[Theorem IV.3.13, Theorem V.2.1]{BG}, there exists a perfect PCAF $\tilde A^{\nu^{(T)}}$ such that
\[
\tilde v^{(T)}_2(s,x)=\mathbb E_{s,x}\int_0^{T-s} d\tilde A^{\nu^{(T)}}_r,\quad (s,x)\in D_T.
\]
Since $\tilde v^{(T)}= v^{(T)}$  a.e.,  we have
\[
\tilde v^{(T)}(s,x)=\mathbb E_{s,x}\varphi(X^D_T)+a\mathbb E_{s,x}\int_0^{T-s}\tilde v^{(T)}(\mathscr X_r)\,dr-\mathbb E_{s,x}\int_0^{T-s}\,d\tilde A^{\nu^{(T)}}_r,\quad (s,x)\in D_T.
\]
Applying now a standard argument (see \cite[Theorem 5.8]{K:JFA}) shows that condition  (iii) of Proposition \ref{prop2.1v} holds for every $(s,x)\in D_T$ with $A^{\nu^{(T)}}$ replaced by $\tilde A^{\nu^{(T)}}$. Assume that $\varphi \in C_0(D)$. Set
\[
\hat h(s,x)=1+\|v\|_{\infty}\mathbb E_{s,x}\int_0^{\infty}1(\mathscr X^{D_0}_r)\,dr.
\]
Since $ h\le\hat h\le \mathbb I_{D\setminus \overline D_0}$,  $v$
is a weak solution to (\ref{eq1.4}) with $\mathbb I_{D\setminus\overline D_0}$ replaced by $\hat h$. Observe that
\[
\hat h(s,x) = 1+\int_0^\infty \int_{D_0}p_{D_0}(r,x,y)\,dy=1+R^{D_0}1(x).
\]
Since $(P^{D_0}_t)$ is strongly Feller (as $\mathbb X$ is strongly Feller), $\hat h\in C(D)$. By probabilistic interpretations of $\tilde v_n$ (cf. \eqref{eq7.wcom1}) and $\tilde v$ (see Proposition \ref{prop2.1v}(2), Proposition \ref{prop.ppt1}(ii)), and \cite[Theorem 1]{Stettner1}, $v\in C([0,T]; C(D))$. Since $v\le w$, and $w \in C([0,T];C_0(D))$, by classical results, $v\in C([0,T];C_0(D))$ as well.
\end{proof}

\subsection{Existence and asymptotics}

\begin{theorem}
\label{th6.1}
\begin{enumerate}
\item[\rm(i)] For every $a\ge 0$ there exists
a unique weak solution $v$ to \mbox{\rm(\ref{eq1.4})}.

\item[\rm(ii)] Let $v_p$, $p>0$, be a weak solution to {\rm{(\ref{eq1.3})}}. Then
for every $\delta\in (0,T]$,
\begin{equation}
\label{eq6.tw}
\sup_{\delta\le t\le T}\|v_p(t)-v(t)\|_\infty+\int_0^T\|v_p(t)-v(t)\|_{H^{\alpha}(D)}\,dt\rightarrow 0
\quad\mbox{\rm as }p\rightarrow\infty.
\end{equation}
\item[\rm(iii)] If, in addition, $\varphi\in H^{\alpha}_0(D)$, then $\frac{d v_p}{d t}\rightarrow \frac{d v}{d t}$ weakly in $L^2(D_T;m_1)$, and if $\varphi\in C_0(D)$, then \mbox{\rm(\ref{eq6.tw})} holds with $\delta=0$.
\end{enumerate}
\end{theorem}
\begin{proof}
Part (i) follows from \cite[Theorem 6.2, Chapter 3]{Lions}. By \cite[Corollary 5.9]{K:JFA},
\begin{equation}
\label{eq6.1}
0\le v_p(s,x)\le e^{sa}\|\varphi\|_\infty.
\end{equation}
By Proposition \ref{prop.ppt1},
\begin{equation}
\label{eq7.kwlezy}
v_p^{(T)}(s,x)=\mathbb E_{s,x}\varphi(X^D_T)+a\mathbb E_{s,x}\int_0^{T-s}v^{(T)}_p(\mathscr X^D_r)\,dr- \mathbb E_{s,x}\int_0^{T-s}b(v_p^{(T)})^p(\mathscr X^D_r)\,dr.
\end{equation}
Therefore, by (\ref{eq6.1}) and \cite[Lemma 94, page 306]{DellacherieMeyer},  there exists a subsequence (still  denoted by $\{v_p\}$) such that
$\{v_p\}$ is convergent a.e. From  this and (\ref{eq6.1}), we infer that  for all $q\ge 1$ and $T\ge 0$, $\{v_p\}$ converges  in $L^q(D_T)$ to some $v\in L^q(D_T)$. Taking $\eta=v_p$ as a test function in (\ref{eq6.0})
we obtain
\[
\|v_p(s)\|^2_{L^2(D;m)}+\int_0^s\EE_D(v_p,v_p)\,dt+\int_0^sbv_p^{p+1}\,dt\nonumber\\
\le \|\varphi\|_{L^2(D;m)}^2+a\int_0^s\|v_p\|^2_{L^2(D;m)}\,dt.
\]
Hence, up to a subsequence, $v_p\rightarrow v$ weakly in $L^2(0,T;H^{\alpha}_0(D))$. Observe also that,  since
$\sup_{p\ge 2}\int_0^sbv_p^{p+1}\,dt<\infty$, we have $v\le \mathbb{I}_{D\setminus \overline D_0}$ a.e. Taking $\eta=\frac{dv_p}{dt}$ in (\ref{eq6.0}) we get
\begin{align}
\label{eq6.3}
\nonumber\int_r^s\Big(\frac{dv_p}{dt}\Big)^2\,dt
&+\frac12\EE_D(v_p(s),v_p(s))+\frac{a}{2}\int_Dv^2_p(r)\,dm+\int_D b\frac{{v_p}^{p+1}(s)}{p+1}\,dm\\&
\quad= \frac12 \EE_D(v_p(r),v_p(r))+\frac{a}{2}\int_Dv^2_p(s)\,dm+\int_D b\frac{{v_p}^{p+1}(r)}{p+1}\,dm.
\end{align}
The rest of the proof we divide into two steps.\\
{\bf Step 1.} We  assume additionally that $\varphi\in H^{\alpha}_0(D)$. Then, by (\ref{eq6.1}) and (\ref{eq6.3}),
\[
\int_0^s\Big(\frac{dv_p}{dt}\Big)^2\,dt+\frac12\EE_D(v_p(s),v_p(s))\le \frac12\EE_D(\varphi,\varphi)+\frac{a}{2}m(D)e^{2s}\|\varphi\|^2_\infty+\frac{\|b\|_\infty}{p+1}.
\]
From this we conclude that $\frac{d v}{d t}\in L^2(D_T;m_1)$ and, up to a subsequence, $\frac{d v_p}{d t}\rightarrow \frac{d v}{d t}$ weakly in $L^2(D_T;m_1)$.
Now, taking $\eta=v_p-v$ as a test function in (\ref{eq6.0}), we find
\begin{align*}
&\|v_p(s)-v(s)\|^2_{L^2(D;m)}+\int_0^s\EE_D(v_p-v,v_p-v)\,dt\\
&\quad \le \int_0^sa(v_p-v,v^p)\,dt-\int_0^s(v_p-v,bv_p^p)\,dt\\
&\qquad-\int_0^s\EE_D(v,v_p-v)\,dt-\int_0^s(\frac{d v}{d t}, v_p-v)\,dt.
\end{align*}
Since $v\le \mathbb{I}_{D\setminus \overline D_0}$ a.e., we have
\begin{align*}
-\int_0^s(v_p-v,bv_p^p)\,dt\le \int_0^s((v-v_p)^+,bv_p^p)\,dt
&\le \int_0^s((v-v_p)^+,bv^p)\,dt\\&
\le \|b\|_\infty\int_0^s|v-v_p|\,dt,
\end{align*}
which converges to zero as $p\rightarrow\infty$. Substituting the above inequality into the previous one
and using already  proven convergences of the sequence $\{v_p\}$, we conclude that
\[
\int_0^T\EE_D(v_p-v,v_p-v)\,dt \rightarrow 0
\]
as $p\rightarrow\infty$. From
this  and a parabolic counterpart of the argument given in (\ref{eq5.tm}) we infer that $v$ is a weak solution to (\ref{eq1.4}). Applying now a uniqueness argument shows the convergence of the whole sequence $\{v_p\}$.

To prove the uniform convergence of $\{v_p\}$ in (\ref{eq6.tw}), we first assume additionally that $\varphi\in C_0(D)$.
Then, since $(P^D_t)_{t\ge0}$ is Fellerian, a fixed point argument shows that  $v_p\in C([0,T];C_0(D))$. By   Proposition \ref{prop.ppt1} and  Proposition \ref{prop2.1vv} (we drop superscript $\tilde\,\,\,$), 
\begin{align*}
(v_p^{(T)}-v^{(T)})(\mathscr X^D_t)&=a\int_t^{T-s}(v^{(T)}_p-v^{(T)})(\mathscr X^D_r)\,dr-\int_t^{T-s}b(v_p^{(T)})^p(\mathscr X^D_r)\,dr\\&\quad +
\int_t^{T-s}\,dA^{\nu^{(T)}}_r-\int_t^{\tau_D}\,d(M_r^p-M_r),\quad t\in [0,T-s],\, P_{s,x}\mbox{-a.s.}
\end{align*}
for  $(s,x)\in D_T$.
By It\^o's formula
\begin{align*}
&|v_p^{(T)}-v^{(T)}|^2(s,x)=2aE_{s.x}\int_0^{T-s}|v^{(T)}_p-v^{(T)}|^2(\mathscr X^D_r)\,dr\\&\quad\quad -2\mathbb E_{s,x}\int_0^{T-s}(v_p^{(T)}-v^{(T)})b(v_p^{(T)})^p(\mathscr X^D_r)\,dr+
2\mathbb E_{s,x}\int_0^{T-s}(v_p^{(T)}-v^{(T)})(\mathscr X^D_r)\,dA^{\nu^{(T)}}_r.
\end{align*}
So, by (H1) and (\ref{eq.tg.2}), there exists $q>1$ such that
\begin{align}
\label{eq7.5}
|v^{(T)}_p(s,x)-v^{(T)}(s,x)|^2&\le 2c(a+\|b\|_\infty)\|v_p-v\|_{L^q(D_T)}
\nonumber\\
&\quad+2\mathbb E_{s,x}\int_0^{T-s}|v^{(T)}_p(\mathscr X^D_r)-v^{(T)}(\mathscr X^D_r)|\,dA^{\nu^{(T)}}_r.
\end{align}
By ultracontractivity of $(P^D_t)_{t\ge 0}$ (cf. \eqref{eq.iuc12u}), for $h>0$, we have
\begin{align*}
&\mathbb E_{s,x}\int_h^{(T-s)\vee h}|v^{(T)}_p(\mathscr X^D_r)-v^{(T)}(\mathscr X^D_r)|\,dA^{\nu^{(T)}}_r\\
&\quad=
\int_D\int_{s+h}^{T\vee(s+h)}|v_p(T-r,y)-v(T-r,y)|p_D(r-s,x,y)\nu^{(T)}(dr\,dy)\\
&\quad\le c\beta_h\frac{e^{-h \lambda_1^D} \|\varphi_1^D\|^2_\infty}{\lambda_1^D} \int_D\int_{s+h}^{T\vee(s+h)}|v_p(T-r,y)-v(T-r,y)|\nu^{(T)}(dr\,dy)\\
&\quad\le c\beta_h\frac{e^{-h\lambda_1^D}  \|\varphi_1^D\|^2_\infty}{\lambda_1^D} \int_D\int_0^T|v_p(r,y)-v(r,y)|\nu(dr\,dy).
\end{align*}
Taking $\eta=|v_p-v|$ as a test function in Proposition \ref{prop2.1v}(ii) (remind here that, as shown above,  $v_p,v\in H^1(0,T;L^2(D;m))$, so $|v_p-v|\in H^1(0,T;L^2(D;m))\subset \WW(0,T)$) and using the already proven convergences of $\{v_p\}$ shows  that
the right-hand side of the above inequality tends to zero as $p\rightarrow \infty$. Next, by (\ref{eq6.1}) and Proposition \ref{prop2.1vv},
\begin{align*}
&\mathbb E_{s,x}\int_0^h|v^{(T)}_p(\mathscr X^D_r)-v^{(T)}(\mathscr X^D_r)|\,dA^{\nu^{(T)}}_r\le
c \mathbb E_{s,x}\int_0^h\,dA^{\nu^{(T)}}_r\\
&\quad= ac\mathbb E_{s,x}\int_0^{h}v^{(T)}(\mathscr X^D_r)\,dr +c\mathbb E_{s,x}v^{(T)}(\mathscr X^D_h)-cv^{(T)}(s,x) \\
&\quad\le ac^2h+
c(\mathbb E_{s,x}v^{(T)}(\mathscr X^D_h)-v^{(T)}(s,x)).
\end{align*}
Observe that
\begin{align*}
\mathbb E_{s,x}v^{(T)}(\mathscr X^D_h)-v^{(T)}(s,x)=P^D_h(v(T-s-h,\cdot))(x)-v(T-s,x).
\end{align*}
Since $v\in C([0,T]; C_0(D))$ (see Proposition \ref{prop2.1vv}), then using the  Feller property of $(P^D_t)$ shows that
\[
\sup_{0\le s\le T}\|P^D_h(v(T-s-h,\cdot))-v(T-s,\cdot)\|_\infty\rightarrow 0 \quad\mbox{as}\quad h\searrow 0.
\]
Since we know that $v_p\rightarrow v $ in $L^q(D_T)$, from (\ref{eq7.5}) and the estimates following it, we deduce that $\|v_p-v\|_\infty\rightarrow 0$ as $p\rightarrow \infty$.

{\bf Step 2.} The general case. Let $\varphi_\varepsilon\in H^{\alpha}_0(D)\cap C_0(D)$ be a positive bounded function such that  $\|\varphi_\varepsilon-\varphi\|_{L_2(D;m)}\le \varepsilon$ and $\varphi_\varepsilon\le\mathbb{I}_{D\setminus\overline D_0}$ a.e. Let $v$ be a weak solution to (\ref{eq1.4}),  $v_p^\varepsilon$ be a weak solution to (\ref{eq1.3})  with $\varphi$
replaced by $\varphi_\varepsilon$, and  $v^\varepsilon$ be a weak solution to (\ref{eq1.4})  with $\varphi$ replaced by $\varphi_\varepsilon$. By a standard argument,
\[
\sup_{t\le T}\|v^\varepsilon_p(t)-v_p(t)\|_{L^2(D;m)}+\Big(\int_0^T\EE_D(v^\varepsilon_p-v_p,v^\varepsilon_p-v_p)\,dt\Big)^{1/2}\le 2e^{a2T}\|\varphi_\varepsilon-\varphi\|_{L^2(D;m)},
\]
\[
\sup_{t\le T}\|v^\varepsilon(t)-v(t)\|_{L^2(D;m)}+\Big(\int_0^T\EE_D(v^\varepsilon-v,v^\varepsilon-v)\,dt\Big)^{1/2}\le 2e^{a2T}\|\varphi_\varepsilon-\varphi\|_{L^2(D;m)}.
\]
Hence
\begin{align*}
&\sup_{t\le T}\|v_p(t)-v(t)\|_{L^2(D;m)}+\Big(\int_0^T\EE_D(v_p-v,v_p-v)\,dt\Big)^{1/2}\\
&\quad \le 4e^{a2T}\varepsilon+
\sup_{t\le T}\|v^\varepsilon_p(t)-v^\varepsilon(t)\|_{L^2(D;m)} +\Big(\int_0^T\EE_D(v^\varepsilon_p-v^\varepsilon,v^\varepsilon_p-v^\varepsilon)\,dt\Big)^{1/2}.
\end{align*}
It follows from this and Step 1 that the second term in (\ref{eq6.tw}) tends to zero as $p\rightarrow \infty$.
For the uniform convergence in (\ref{eq6.tw}), let $v_n$ be a weak solution of (\ref{eq7.wcom1}) and $v^\varepsilon_n$
be a weak solution of  (\ref{eq7.wcom1}) with $\varphi$ replaced by $\varphi_\varepsilon$. Set $w^\varepsilon_n= v_n-v^\varepsilon_n$.
Observe that 
\[
\frac{d w^\varepsilon_n}{d t}-(\Delta^{\alpha})_{|D} w^\varepsilon_n=aw^\varepsilon_n+F_n(\cdot, w^\varepsilon_n),\quad w^\varepsilon_n(0)=\varphi-\varphi_\varepsilon,
\]
where $F_n(t,y)=-n(y+v^\varepsilon_n-\mathbb{I}_{D\setminus\overline D_0})^++n(v^\varepsilon_n-\mathbb{I}_{D\setminus\overline D_0})^+$.
By Proposition \ref{prop.ppt1}(ii) for q.e. $(s,x)\in D_T$ and any  $t\in [0,T-s]$,
\begin{align*}
w^{\varepsilon,(T)}_n(\mathscr X^D_t)&=E\Big((\varphi-\varphi_\varepsilon)(X^D_T)+a\int_t^{T-s} w^{\varepsilon,(T)}_n(\mathscr X^D_r)\,dr\\&\quad\quad\quad+\int_t^{T-s} F_n^{(T)}(\mathscr X^D_r,w^{\varepsilon,(T)}_n(\mathscr X^D_r))\,dr\Big|\FF_t\Big).
\end{align*}
From this and \cite[Lemma 2.3]{KR:JFA} (see also Remark \ref{rem.bsde}), we deduce that
\[
|v_n^\varepsilon(T-s,x)-v_n(T-s,x)|\le e^{aT}\mathbb E_{s,x}|\varphi_\varepsilon(X^D_T)-\varphi(X^D_T)|.
\]
By Proposition \ref{prop2.1v},  and (\ref{eq.tg.1}), for every $s\in [0,T-\delta]$,
\begin{align*}
|v^\varepsilon(T-s,x)-v(T-s,x)|&\le e^{aT}\mathbb E_{s,x}|\varphi_\varepsilon(X^D_T)-\varphi(X^D_T)|\\
&=e^{aT}\int_D|\varphi_\varepsilon(y)-\varphi(y)|p_D(T-s,x,y)\,dy\\
&\le c\frac{e^{aT}}{\delta^{d/2\alpha}}\int_D|\varphi_\varepsilon(y)-\varphi(y)|\,dy.
\end{align*}
Analogously, we get the above estimate for $|v_p^\varepsilon(T-s,x)-v_p(T-s,x)|$. From this and Step 1, we get the desired result.
\end{proof}

\section{Asymptotics as $t\rightarrow \infty$}

As in Sections 3-7, we assume that  the hypotheses (H1), (H2) are satisfied  and $D$ is a bounded Lipschitz domain.

\subsection{Cauchy-Dirichlet problem}

\begin{lemma}
\label{lm7.1}
Let $v_p$ be a weak solution to {\rm (\ref{eq1.3})}. Assume that there exists $c>0$ such that $b\ge c$ on $D$. Then
\[
v_p(s,x)\le \max\{\|\varphi\|_\infty,(a/c)^{1/(p-1)}\},\quad x\in D,\, s\ge 0.
\]
\end{lemma}
\begin{proof}
Set $M_p:=\max\{\|\varphi\|_\infty,(a/c)^{1/(p-1)}\}$. By Proposition \ref{prop.ppt1} and the Tanaka-Meyer formula (see, e.g., \cite[IV.Theorem 70]{Protter})
\begin{align*}
&(v^{(T)}_p(s,x)-M_p)^+\le \mathbb E_{s,x}(\varphi(X^D_T)-M_p)^+\\
&\qquad+\mathbb E_{s,x}\int_0^{T-s}\mathbf{1}_{\{v^{(T)}_p(\mathscr X^D_r)>M_p\}}
\Big(av^{(T)}_p(\mathscr X^D_r)-(b(v^{(T)}_p)^p)(\mathscr X^D_r)\Big)\,dr\le 0,
\end{align*}
the last inequality being a consequence of the fact that $ay-cy^p\le 0$, $y\ge M_p$.
\end{proof}
\begin{corollary}
Under the assumptions of Lemma \ref{lm7.1}, the reaction measure $\mu$ for the unique weak solution $v$ to \mbox{\rm(\ref{eq1.4})} is of the form
$\mu=g\cdot m_1$ with some positive $g\in L^\infty(\mathbb R^+\times D)$.
\end{corollary}
\begin{proof}
By Lemma \ref{lm7.1}, the term $bv_p^p$ in \eqref{eq1.3} is bounded uniformly in $p\ge 1$.  Therefore, applying  Theorem \ref{th6.1} gives the result.
\end{proof}

\begin{lemma}
\label{lm7.2}
Let $v_p$ be a weak solution to \mbox{\rm (\ref{eq1.3})} and  $a\in (\lambda_1^D,\lambda_1^{D_0})$. Then for every $\delta, t_0>0$ there exist $M, c>0$ such that
\begin{equation}
\label{eq7.1}
v_p(t,x)\le M,\quad (t,x)\in \mathbb R^+\times D,\quad p\ge 1+\delta,
\end{equation}
and
\begin{equation}
\label{eq7.2}
c\varphi^D_1(x)\le v_p(t,x),\quad t\ge t_0,\, x\in D,\quad p\ge 1+\delta.\qquad
\end{equation}
\end{lemma}
\begin{proof}
For $\varepsilon>0$, we set $D_\varepsilon=\{x\in \BR^d: \mbox{dist}\{x,D\}<\varepsilon\}$.
Let $\{q_k\}$ be a sequence of bounded positive functions on $D_\varepsilon$ such that
supp$[q_k]\subset D_\varepsilon\setminus \overline D_0$ and for every compact $K\subset D_\varepsilon\setminus \overline D_0$, $\inf_{x\in K}q_k(x)\nearrow \infty$. Let  $\psi_k$ be the principal eigenfunction
for $-(\Delta^{\alpha})_{|D_\varepsilon}+q_k$. As in the proof of Theorem \ref{th5.1} we show   that for fixed $\delta>0$ there exist  $k_0\in\mathbb N$ and $c_{k_0}>0$ such that $\lambda^D_1[q_k]\ge a,\, k\ge k_0$, $c\psi_{k_0}\ge\varphi^{D_\varepsilon}_1$ and $c\psi_{k_0}$ is a supersolution to (\ref{eq1.1}) on $D_\varepsilon$ for $c\ge c_{k_0},\, p\ge 1+\delta$. More precisely, there exists a positive bounded function $h$ on $D_\varepsilon$ such that
\[
-(\Delta^{\alpha})_{|D_\varepsilon}(c\psi_{k_0})=a(c\psi_{k_0})-b(c\psi_{k_0})^p+h
\]
(see the reasoning following (\ref{eq5.1.1.2})).
Of course, since $\psi_{k_0}$ is independent of $t$, we have
\[
\frac{d(c\psi_{k_0})}{dt}-(\Delta^{\alpha})_{|D_\varepsilon}(c\psi_{k_0})=a(c\psi_{k_0})-b(c\psi_{k_0})^p+h.
\]
Let $c$ be chosen so that  $c\psi_{k_0}\ge 1$ on $D$.  By Proposition \ref{prop.ppt1} and the Tanaka-Meyer formula, for every $t\in[0,T-s]$,
\begin{align*}
&\mathbb E_{s,x}(v^{(T)}_p(\mathscr X^D_t)-c\psi_{k_0}(\mathscr X^{D_\varepsilon}_t))^+\\
&\quad\le \mathbb E_{s,x}(\varphi(X^D_T)-c\psi_{k_0}(X^{D_\varepsilon}_T))^+\\
&\qquad+a\mathbb E_{s,x}\int_t^{T-s}\mathbf{1}_{\{v^{(T)}_p(\mathscr X^D_r)>c\psi_{k_0}(\mathscr X^{D_\varepsilon}_r)\}}(v^{(T)}_p(\mathscr X^D_r)-c\psi_{k_0}
(\mathscr X^{D_\varepsilon}_r))\,dr\\
&\qquad
-\mathbb E_{s,x}\int_t^{T-s}\mathbf{1}_{\{v^{(T)}_p(\mathscr X^D_r)>c\psi_{k_0}(\mathscr X^{D_\varepsilon}_r)\}}b((v^{(T)}_p)^p(\mathscr X^D_r)-(c\psi_{k_0})^p(\mathscr X^{D_\varepsilon}_r))\,dr\\
&\qquad- \mathbb E_{s,x}\int_t^{T-s}\mathbf{1}_{\{v^{(T)}_p(\mathscr X^D_r)>c\psi_{k_0}
(\mathscr X^{D_\varepsilon}_r)\}}h(\mathscr X^{D_\varepsilon}_r)\,dr
\\& \quad
\le  \mathbb E_{s,x}(\varphi(X^D_T)-c\psi_{k_0}(X^{D_\varepsilon}_T))^+
+a\mathbb E_{s,x}\int_t^{T-s}(v^{(T)}_p(\mathscr X^D_r)-c\psi_{k_0}(\mathscr X^{D_\varepsilon}_r))^+\,dr.
\end{align*}
Applying Gronwall's lemma gives
\[
\mathbb E_{s,x}(v^{(T)}_p(\mathscr X^D_t)-c\psi_{k_0}(\mathscr X^{D_\varepsilon}_t))^+\le e^{aT} \mathbb E_{s,x}(\varphi(X^D_T)-c\psi_{k_0}(X^{D_\varepsilon}_T))^+,\quad t\in [0,T-s].
\]
Observe that
\begin{align*}
 \mathbb E_{s,x}(\varphi(X^D_T)-c\psi_{k_0}(X^{D_\varepsilon}_T))^+&=
\mathbb E_{s,x}\mathbf{1}_{\{\varphi(X^D_T)>0\}}(\varphi(X^D_T)-c\psi_{k_0}(X^{D_\varepsilon}_T))^+\\&=
\mathbb E_{s,x}\mathbf{1}_{\{\varphi(X^D_T)>0\}}(\varphi(X_T)-c\psi_{k_0}(X_T))^+=0.
\end{align*}
The last equality follows from the fact that $c\psi_{k_0}\ge 1$ on $D$. Consequently,
\[
\mathbb E_{s,x}(v^{(T)}_p(\mathscr X^D_t)-c\psi_{k_0}(\mathscr X^D_t))^+=0,\quad t\in [0,T-s].
\]
Taking $t=0$, we get (\ref{eq7.1}). Now, let $w_p$ be a weak solution to the Cauchy problem
\[
\frac{d w_p}{d t}-(\Delta^{\alpha})_{|D}w_p=aw_p-\|b\|_\infty w_p^p,\quad w_p(0)=\varphi.
\]
By \cite[Corollary 5.9]{K:JFA}, $w_p\le v_p$. By Lemma \ref{lm7.1}, $h:=\|b\|_\infty w^{p-1}_p$ is bounded by a constant independent of $p$. Observe that
\[
\frac{d w_p}{d t}-(\Delta^{\alpha})_{|D}w_p+hw_p=aw_p,\quad w_p(0)=\varphi.
\]
By \cite[Corollary 5.9]{K:JFA}, $w_p\ge \bar w_p$, where $\bar w_p$ is a weak solution to the problem
\[
\frac{d\bar w_p}{d t}-(\Delta^{\alpha})_{|D}\bar w_p+\|h\|_\infty \bar w_p=a\bar w_p,\quad \bar w_p(0)=\varphi.
\]
By the Feynman-Kac formula,
\[
\bar w_p(t,x)=\int_De^{(a-\|h\|_\infty)t}p_D(t,x,y)\,\varphi(y)\,dy.
\]
By the ultracontractivity of $(P^D_t)$ (cf. \eqref{eq.iuc12l})
\[
\bar w_p(t_0,x)\ge \beta_{t_0} e^{-t_0\lambda_1^D}e^{(a-\|h\|_\infty)t_0} \varphi^D_1(x)\int_D\varphi^D_1(y)\varphi(y)\,dy=c_{t_0}   \varphi^D_1(x),
\]
with some $c_{t_0}>0$. Now, observe that for any $c>0$,
\[
\frac{d (c\varphi^D_1)}{d t}-(\Delta^{\alpha})_{|D}(c\varphi^D_1)=a(c\varphi^D_1)-\|b\|_\infty(c\varphi^D_1)^p+
\{\|b\|_\infty(c\varphi^D_1)^p+(\lambda_1^D-a)c\varphi^D_1\}.
\]
For $c\le (a-\lambda_1^D)^{p-1}/(\|b\|_\infty\|\varphi^D_1\|_\infty)$,
the last  term in braces on the right-hand side of the above equation  is less than or equal to zero.
Therefore, by \cite[Corollary 5.9]{K:JFA}, for such $c>0$ we have $w_p(t)\ge (c\wedge c_{t_0})\varphi^D_1$, $t\ge t_0$.
This completes the proof since $v_p\ge w_p$ as shown above.
\end{proof}

The  following simple lemma appears to  be very useful in the proofs of the large time asymptotics of solutions to \eqref{eq1.3}, \eqref{eq1.4} (see e.g.  \cite{Zheng} for the similar technique).

\begin{lemma}
\label{lm.vulem}
Let  $T,a>0$, and $c\in (0,T)$. Assume that $y:(0,\infty)\to [0,\infty)$, $\int_0^\infty y(r)\,dr<\infty$, and  
\begin{equation}
\label{eq.vul1}
y(t)-y(s)\le a\int_s^ty(r)\,dr,\quad 0< s\le t<T.
\end{equation} 
Then 
\begin{equation}
\label{eq.vul3}
y(t)\le e^{\int_0^Ty(r)\,dr}\big(c^{-1}\int_0^Ty(r)\,dr+a^2c\big),\quad t\in [c,T)\,\,\, a.e.
\end{equation} 
Moreover, if $\int_0^\infty y(r)\,dr<\infty$, and \eqref{eq.vul1} holds with $T=\infty$, then 
$\lim_{t\to \infty}y(t)=0$.
\end{lemma} 
\begin{proof}
We first prove the second assertion. By \eqref{eq.vul1} and integrability assumption on $y$, we have that $c_\delta:=\sup_{t\ge \delta}y(t)<\infty$ for any $\delta>0$.
Suppose that $\limsup_{t\to \infty}y(t)=\beta>0$. Set $T:= \frac{\beta}{4ac}$. Then there exists an increasing sequence $\{t_n\}$ ($t_1\ge \delta$)
such that $t_{n+1}-t_n\ge 2T$, and $y(t_n)\ge \beta/2,\, n\ge 1$. Consequently, for any $t\in [t_{n+1}-T,t_{n+1}]$,
\[
\beta/2-y(t)\le y(t_{n+1})-y(t)\le a\int_t^{t_{n+1}}y(r)\,dr\le ac(t_{n+1}-t)\le acT=\beta/4.
\]
This in turn implies that
\[
\beta/4\le y(t),\quad t\in [t_{n+1}-T,t_{n+1}],\, n\ge 1,
\]
which  contradicts to integrability of $y$ over $(0,\infty)$.

As to the second assertion, we let $j:\BR\to\BR$ be a positive smooth function with compact support in $B(0,1)$
such that $j(0)=1$, and $\int_{\BR}j(t)\,dt=1$. Set $j_\varepsilon(t):=\frac{1}{\varepsilon}j(t/\varepsilon),\, t\in\BR$,
and $\hat y_\varepsilon(t):=y_\varepsilon(t+\varepsilon)$, with
\[
y_\varepsilon(t):= \int_\BR j_\varepsilon(t-s)y(s)\,ds,\quad t\in (\varepsilon,T-\varepsilon).
\]
Using this notation and \eqref{eq.vul1}, we find that 
\[
\hat y_\varepsilon(t)-\hat y_\varepsilon(s)\le a\int_s^t\hat y_\varepsilon (r)\,dr,\quad 0< s\le t<T-\varepsilon,
\]
and hence in turn that 
\[
\frac{d \hat y_\varepsilon }{d t}(t)\le a \hat y_\varepsilon(t)\le  \hat y^2_\varepsilon(t)+a^2,\quad t\in (0,T-\varepsilon).
\]
By \cite[Lemma 1.1]{Zheng},
\[
\hat y^2_\varepsilon(t)\le e^{\int_0^{T-\varepsilon}\hat y_\varepsilon(r)\,dr}\big(c^{-1}\int_0^{T-\varepsilon}\hat y^2_\varepsilon(r)\,dr+a^2c\big),\quad t\in (c,T-\varepsilon).
\]
Equivalently,
\[
y^2_\varepsilon(t)\le e^{\int_\varepsilon^{T} y_\varepsilon(r)\,dr}\big(c^{-1}\int_\varepsilon^{T} y_\varepsilon(r)\,dr+a^2c\big),\quad t\in (c+\varepsilon,T).
\]
Now, letting $\varepsilon\searrow 0$, we easily get the desired result.
\end{proof}

\begin{proposition}\label{prop8.5}
Let $v_p$ be a weak solution to {\rm{(\ref{eq1.3})}} and  $u_p$ be a weak solution to {\rm{(\ref{eq1.1})}}. Then
\begin{equation}
\label{eq8.3}
\|v_p(t)-u_p\|_\infty +\|v_p(t)-u_p\|_{H^{\alpha}(D)}\rightarrow 0\quad \mbox{as }t\rightarrow \infty.
\end{equation}
\end{proposition}
\begin{proof}
Let $j:\BR\to\BR$ be as in the proof of Lemma \ref{lm.vulem}.
For any $u\in L^1_{loc}(D_T)$ (extended by zero on $D_T^c$) we denote
\[
u^{(\varepsilon)}(t):=\int_{\BR}j_\varepsilon(t-s)u(s)\,ds.
\]
Observe that (cf. Theorem \ref{th6.impar5})
\[
\frac{d^2 v^{(\varepsilon)}_p}{d t^2}-(\Delta^{\alpha})_{|D}\frac{d v^{(\varepsilon)}_p}{d t}=a \frac{d v^{(\varepsilon)}_p}{d t}-pb\big[v_p^{p-1}\frac{d v_p}{d t}\big]^{(\varepsilon)}.
\]
Multiplying the above equation by $\frac{d v^{(\varepsilon)}_p}{d t}$ and integrating over $D$ we find
\begin{align*}
&\frac12\frac{d}{dt}\Big\|\frac{d v^{(\varepsilon)}_p}{d t}(t)\Big\|^2_{L^2(D;m)}+\EE_D(\frac{d v^{(\varepsilon)}_p}{d t}(t),\frac{d v^{(\varepsilon)}_p}
{d t}(t))\\
&\qquad +\int_Db\big[v^{p-1}(t)\frac{d v_p}{d t}\big]^{(\varepsilon)}(t)\frac{d v^{(\varepsilon)}_p}{d t}(t)\, dm
=a\Big\|\frac{d v^{(\varepsilon)}_p}{d t}(t)\Big\|^2_{L^2(D;m)},\quad t\in [\varepsilon,T-\varepsilon].
\end{align*}
Integrating over $[t_1,t_2]\subset [\varepsilon,T-\varepsilon]$, and letting $\varepsilon\searrow 0$ yields 
\begin{align*}
\Big\|\frac{d v_p}{d t}(t_2)\Big\|^2_{L^2(D;m)}-\Big\|\frac{d v_p}{d t}(t_1)\Big\|^2_{L^2(D;m)}
\le 2a\int_{t_1}^{t_2}\Big\|\frac{d v_p}{d t}(t)\Big\|^2_{L^2(D;m)}\,dt.
\end{align*}
Write $y_p(t)=\|\frac{d v_p}{d t}(t)\|^2_{L^2(D;m)} $. Then
\begin{equation}
\label{eq7.3}
y_p(t_2)-y_p(t_1)\le 2a\int_{t_1}^{t_2}y_p(t)\,dt,\quad t_2\ge t_1>0.
\end{equation}
Since $v_p(t)\in H^{\alpha}_0(D)$ for $t>0$, we conclude from  (\ref{eq6.3}) and Lemma \ref{lm7.2} that  $\int_\varepsilon^\infty y_p(t)\,dt<\infty$ for any $\varepsilon>0$. Therefore, by  Lemma \ref{lm.vulem}, $y_p(t)\rightarrow 0$.  By
(\ref{eq6.3}) and Lemma \ref{lm7.2} again,
$\sup_{t\ge t_0}\EE_D(v_p(t),v_p(t))<\infty$ for every $t_0>0$. As a result,  there exists a sequence $\{t_n\}\subset \BR^+$ such that $t_n\rightarrow \infty$ and  $\{v_p(t_n)\}$
is convergent as $n\rightarrow\infty$ to some $w_p\in H^{\alpha}_0(D)$ weakly in $H^{\alpha}(D)$ and strongly in $L^q(D;m)$ for any $q\ge 1$.
Since $v_p$ is a weak solution to (\ref{eq1.3}), we have (cf. Remark \ref{rem.th6.impar5})
 \[
 (\frac{d v_p}{d t}(t_n),\eta)+\EE_D(v_p(t_n),\eta)=a\int_Dv_p(t_n)\eta\,dm-\int_D bv_p^p(t_n)\eta\,dm
 \]
 for any $\eta\in H^{\alpha}_0(D)$. Letting $n\rightarrow \infty$ we obtain
 \[
\EE_D(w_p,\eta)=a\int_Dw_p\eta\,dm-\int_Db w_p^p\eta\,dm
 \]
for any $\eta\in H^{\alpha}_0(D)$. By Lemma \ref{lm7.2}, $w_p$ is strictly positive, and so by Proposition \ref{prop3.2}, $w_p=u_p$. By the  uniqueness argument, the convergence of $\{v_p(t_n)\}$ can be strengthened to the convergence of $v_p(t)$ as $t\rightarrow \infty$. Now, subtracting  (\ref{eq6.0a}) from (\ref{eq5.0}), and  then taking $\eta=v_p(t)-u_p$
as a test function, and using the  already proven convergences of $\{v_p(t)\}$ as $t\rightarrow\infty$, we deduce   that $v_p(t)\rightarrow u_p$ in $H^{\alpha}(D)$ as $t\rightarrow\infty$.

To prove the uniform convergence  in (\ref{eq8.3}), we first observe that
\[
\frac{d (v_p-u_p)}{d t}-(\Delta^{\alpha})_{|D}(v_p-u_p)=a(v_p-u_p)-(bv_p^p-bu_p^p),\quad (v_p-u_p)(0)=\varphi-u_p.
\]
Consequently, by Proposition \ref{prop.ppt1} and the Tanaka-Meyer formula, for every $x\in D$,
\[
|v_p^{(T)}(0,x)-u_p^{(T)}(0,x)|\le \mathbb E_{0,x}|\varphi(X^D_T)-u_p(X^D_T)|+a\mathbb E_{0,x}\int_0^T|v^{(T)}_p-u^{(T)}_p|(\mathscr X^D_r)\,dr.
\]
Equivalently,
\[
|v_p(T,x)-u_p(x)|\le \mathbb E_{0,x}|\varphi(X^D_T)-u_p(X^D_T)|+a\mathbb E_{0,x}\int_0^T|v_p-u_p|(T-r, X^D_r)\,dr
\]
for every $x\in D$. By (\ref{eq.tg.1}) and \eqref{eq7.1},
\[
E_{0,x}|\varphi(X^D_T)-u_p(X^D_T)|\le \hat cMT^{-d/2\alpha},\quad x\in D.
\]
Let $0< S<T$. We have
\begin{align*}
\mathbb E_{0,x}\int_0^T|v_p-u_p|(T-r, X^D_r)\,dr&=\int_D\int_0^S|v_p-u_p|(T-r,y)p_D(r,x,y)\,dy\\
&\quad+\int_D\int_S^T|v_p-u_p|(T-r,y)p_D(r,x,y)\,dy.
\end{align*}
By    (\ref{eq7.1}) and ultracontractivity of $(P^D_t)_{t\ge0}$ (cf. \eqref{eq.iuc12u}),
\[
\int_D\int_S^T|v_p-u_p|(T-r,y)p_D(r,x,y)\,dy\le \beta_Sm(D)\frac{2M\|\varphi_1^D\|^2_\infty}{\lambda_1^D}e^{-\lambda_1^DS}.
\]
By (\ref{eq.estgfu2}), for any $q\in [1,d/(d-2\alpha))$
\[
\int_D\int_0^S|v_p-u_p|(T-r,y)p_D(r,x,y)\,dy\le \sup_{x\in D}\|G(x,\cdot)\|_{L^q(D;m)}\sup_{T-S\le r\le T} \|v_p(r)-u_p\|_{L^{q*}(D;m)}.
\]
Summing up the above inequalities, we conclude that  for any $x\in D$,
\begin{align}
\label{eq7.uc}
\nonumber |v_p(T,x)-u_p(x)|&\le \hat cMT^{-d/2\alpha}+\beta_Sm(D)\frac{2M\|\varphi_1^D\|^2_\infty}{\lambda_1^D}e^{-\lambda_1^DS}\\
&\quad+ \sup_{x\in D}\|G(x,\cdot)\|_{L^q(D;m)}\sup_{T-S\le r\le T}\|v_p(r)-u_p\|_{L^{q^*}(D;m)}.
\end{align}
Since  $v_p(t)\rightarrow u_p$ in $L^{q^*}(D;m)$ as shown above, we conclude from (\ref{eq7.uc}), by letting $T\to \infty$
and then $S\to \infty$, that  $\|v_p(T)-u_p\|_{\infty}\rightarrow0$ as $T\rightarrow\infty$.
\end{proof}

\subsection{Obstacle problem}

\begin{theorem}
Let $\varphi\in H^{\alpha}_0(D)$, $v$ be a weak solution to \mbox{\rm(\ref{eq1.4})} and $u$ be a weak solution to {\rm{(\ref{eq1.2})}}. Then
\begin{equation}
\label{eq8.6}
\|v(t)-u\|_\infty+\|v(t)-u\|_{H^{\alpha}(D)}\rightarrow 0\quad\mbox{as }t\rightarrow \infty.
\end{equation}
\end{theorem}
\begin{proof}
We divide the proof into two steps.\\
{\bf Step 1.} Convergence in the energy norm. By (\ref{eq6.3}) and (\ref{eq7.1}),
\begin{align}
\label{eq7.4}
\int_0^s\Big(\frac{dv_p}{dt}\Big)^2\,dt+\frac12\EE_D(v_p(s),v_p(s))\le  \frac12 \EE_D(\varphi,\varphi)+\frac{a}{2}M^2 m(D)+\|b\|_\infty m(D).
\end{align}
Therefore, under the notation of the prof of Proposition \ref{prop8.5},
\begin{equation}
\label{eq.vulem4}
\sup_{p\ge 1}\int_0^\infty y_p(r)\,dr<\infty.
\end{equation}
By \eqref{eq7.3} and Lemma \ref{lm.vulem} for any $\delta>0$,
\begin{equation}
\label{eq.vulem5}
\sup_{p\ge 1}\sup_{t\ge \delta}y_p(t)=:d_\delta<\infty.
\end{equation}
Set $k_p(t):=-y_p(t)+2a\int_0^ty_p(r)\,dr,\, t>0$. By \eqref{eq7.3}, $k_p$ is non-decreasing, and
 \begin{equation}
\label{eq.vulem6}
y_p(t)-y_p(s)=2a\int_s^ty_p(r)\,dr-\int_s^t\,dk_p(r),\quad 0<s<t.
\end{equation}
By \eqref{eq.vulem4}, \eqref{eq.vulem5}, there exist $y\in L^1(0,\infty)$ and non-increasing right-continuous
function $k$ such that
\[
\int_0^ty_p(r)\,dr\to\int_0^ty(r)\,dr,\quad k_p(t)\to k(t),\quad \mbox{a.e.}\,\,\, t>0
\]
Set $\hat y(t):= 2a\int_0^ty(r)\,dr-k(t),\, t>0$. Clearly, $\hat y(t)=\limsup_{p\to \infty}y_p(t)=y(t)$ a.e. $t>0$.
Thus,
 \begin{equation}
\label{eq.vulem7}
\hat y(t)-\hat y(s)=2a\int_s^t\hat y_p(r)\,dr-\int_s^t\,dk_p(r),\quad 0<s<t.
\end{equation}
Hence
\[
\hat y(t)-\hat y(s)\le 2a\int_s^t\hat y(r)\,dr,\quad 0<s<t.
\]
By \eqref{eq.vulem4}, $\hat y\in L^1(0,\infty)$, and so, by Lemma \ref{lm.vulem},
 \begin{equation}
\label{eq.vulem8}
\hat y(t)\to 0,\quad t\to\infty.
\end{equation}
By Theorem \ref{th6.1}(iii),
\[
\int_s^t\int_D\big(\frac{d v}{d t}\big)^2\,dm_1\le\liminf_{p\to\infty} \int_s^t\int_D\big(\frac{d v_p}{d t}\big)^2\,dm_1=\liminf_{p\to\infty}\int_s^ty_p(r)\,dr=\int_s^t\hat y(r)\,dr.
\]
Thus,
 \begin{equation}
\label{eq.vulem9}
\|\frac{d v}{d t}(t)\|^2_{L^2(D;m)}\le \hat y(t),\quad \mbox{a.e.}\,\,\, t>0.
\end{equation}
By Theorem \ref{th6.1} and (\ref{eq7.4}),
$\mbox{ess\,sup}_{t\ge 0}\EE_D(v(t),v(t))<\infty.$
Consequently,  there exist a sequence $\{t_n\}$, and $u\in H^{\alpha}_0(D)$ such that $t_n\rightarrow \infty$, $\frac{d v} {d t}(t_n)\rightarrow 0$
in $L^2(D;m)$, $v(t_n)\rightarrow u$ weakly in $H^{\alpha}_0(D)$ and for every $q\ge1$, $v(t_n)\rightarrow u$ in $L^q(D;m)$.
We may assume that $\{t_n\}$ is chosen so that (cf. Remark \ref{rem.eqovsop1})
\[
\Big(\frac{d v}{d t}(t_n),\eta-v(t_n)\Big)_{L^2(D:m)}+\EE_D(v(t_n),\eta-v(t_n))\ge a(v(t_n),\eta-v(t_n))_{L^2(D;m)}
\]
for every $\eta\in H^{\alpha}_0(D)$ such that $\eta\le\mathbb{I}_{D\setminus\overline D_0}$ a.e. Letting $n\rightarrow\infty$
and using (\ref{eq7.2}) shows that $u$ is a weak solution to (\ref{eq1.2}). By the uniqueness for (\ref{eq1.2}) (see \cite{K:arx1}) and the fact that $v\in \WW(0,T)\subset C([0,T]; L^2(D;m))$, $T\ge 0$, we have $v(t)\rightarrow u$ as $t\rightarrow \infty$ strongly in $L^2(D;m)$. By (\ref{eq6.0a}),
\begin{align*}
&\EE_D(v_p(t)-v_p(s),v_p(t)-v_p(s))\\
&\quad\le a\|v_p(t)-v_p(s)\|^2_{L^2(D;m)}
+\Big\|\frac{d v_p}{d t}(t)-\frac{d v_p}{d t}(s)\Big\|_{L^2(D;m)}\|v_p(t)-v_p(s)\|_{L^2(D;m)}.
\end{align*}
From  (\ref{eq.vulem5}) and Theorem \ref{th6.1}, we infer that $v\in C((0,T]; H^{\alpha}_0(D))$, and  for all $s,t\ge 1$,
\begin{align*}
\EE_D(v(t)-v(s),v(t)-v(s))\le a\|v(t)-v(s)\|^2_{L^2(D;m)}+2d_1\|v(t)-v(s)\|_{L^2(D;m)}.
\end{align*}
From this and already proven properties of $v$,  we conclude that  $v(t)\rightarrow u$ in $H^{\alpha}_0(D)$ as $t\rightarrow \infty$.

{\bf Step 2.} The uniform convergence in (\ref{eq8.6}). By Theorem \ref{th5.1} and Theorem \ref{th6.1}, letting $p\rightarrow\infty$ in (\ref{eq7.uc}) we get
\begin{align*}
|v(T,x)-u(x)|&\le \hat cMT^{-d/2\alpha}+\beta_Sm(D)\frac{2M\|\varphi_1^D\|_\infty}{\lambda_1^D}e^{-S}\\
&\quad+
\sup_{x\in D}\|G_D(x,\cdot)\|_{L^{q}(D;m)}\sup_{T-S\le r\le T}\|v(r)-u\|_{L^{q^*}(D;m)},\quad x\in D.
\end{align*}
for $x\in D$. By the asymptotics proved in the first step, we conclude at once  from the above inequality  that $\|v(T)-u\|_\infty\rightarrow 0$ as $T\rightarrow \infty$.
\end{proof}

\subsection*{Acknowledgements}
{\small This work was supported by Polish National Science Centre
(Grant No. 2017/25/B/ST1/00878).}

\end{document}